\pdfoutput=1
\documentclass[10pt]{amsart}

\usepackage{amsmath,amscd}
\usepackage{amssymb}
\usepackage{mathtools}
\usepackage{stmaryrd}
\usepackage{url}
\usepackage{tikz-cd}

\usepackage{times}

\usepackage{tikz}

\usepackage{enumitem,kantlipsum}

\author{Liran Shaul}
\address{Department of Algebra, Faculty of Mathematics and Physics, Charles University in Prague, Sokolovsk\'a 83, 186 75 Praha, Czech Republic}

\email{shaul@karlin.mff.cuni.cz}
\newtheorem{thm}[equation]{Theorem}
\newtheorem*{thm*}{Theorem}
\newtheorem*{cor*}{Corollary}
\newtheorem*{dfn*}{Definition}

\newtheorem{cthm}{Theorem}
\newtheorem{cor}[equation]{Corollary}
\newtheorem{prop}[equation]{Proposition}
\newtheorem{lem}[equation]{Lemma}

\theoremstyle{definition}
\newtheorem{dfn}[equation]{Definition}
\newtheorem{rem}[equation]{Remark}
\newtheorem{exa}[equation]{Example}

\newcommand{\opn}{\operatorname}
\newcommand{\cat}[1]{\operatorname{\mathsf{#1}}}

\newcommand{\mfrak}[1]{\mathfrak{#1}}

\newcommand{\mrm}[1]{\mathrm{#1}}
\newcommand{\mbb}[1]{\mathbb{#1}}

\renewcommand{\k}{\Bbbk}

\newcommand{\K}{\mbb{K} \hspace{0.05em}}

\newcommand{\m}{\mfrak{m}}
\newcommand{\n}{\mfrak{n}}
\newcommand{\p}{\mfrak{p}}
\newcommand{\q}{\mfrak{q}}

\newcommand{\amp}{\operatorname{amp}}
\def\skewtimes{\ltimes\!}

\usepackage{microtype}

\begin{document}

\title{Sequence-regular commutative DG-rings}

\begin{abstract}
We introduce a new class of commutative noetherian DG-rings which generalizes the class of regular local rings.
These are defined to be local DG-rings $(A,\bar{\m})$ such that the maximal ideal $\bar{\m} \subseteq \mrm{H}^0(A)$ can be generated by an $A$-regular sequence.
We call these DG-rings sequence-regular DG-rings, and make a detailed study of them.
Using methods of Cohen-Macaulay differential graded algebra, 
we prove that the Auslander-Buchsbaum-Serre theorem about localization generalizes to this setting.
This allows us to define global sequence-regular DG-rings, 
and to introduce this regularity condition to derived algebraic geometry.
It is shown that these DG-rings share many properties of classical regular local rings,
and in particular we are able to construct canonical residue DG-fields in this context.
Finally, we show that sequence-regular DG-rings are ubiquitous,
and in particular, any eventually coconnective derived algebraic variety over a perfect field is generically sequence-regular.
\end{abstract}


\setcounter{tocdepth}{1}
\setcounter{section}{-1}

\maketitle
%\tableofcontents

\numberwithin{equation}{section}

\section{Introduction}

The aim of this paper is to introduce and study a class of commutative differential graded rings which generalizes the class of regular local rings, and the corresponding global notion of a regular ring.

The class of regular local rings is quite possibly the most important class of rings in commutative algebra.
This is because of the geometric interpretation of regularity:
if $X$ is a noetherian scheme and $p\in X$,
then $X$ is nonsingular at $p$ if and only if the local ring $\mathcal{O}_{X,p}$ is a regular local ring.
Thus, regularity is the standard situation in algebraic geometry,
while all other local rings pertain only to special geometric situations.

There are various equivalent definitions for the notion of a regular local ring.
Since in this paper we are dealing with differential graded rings,
which are homological in nature, 
the most tempting definition one would consider using is that a noetherian local ring is regular precisely when it has finite global dimension.

Unfortunately, under an eventual coconnectedness assumption 
(which means that the cohomology or the homotopy of objects is bounded), 
as was shown by J{\o}rgensen in \cite[Theorem 0.2]{Jo}, by Yekutieli in \cite[Theorem 0.7]{YeDual}
and by Lurie in \cite[Lemma 11.3.3.3]{LurSpec} in different derived contexts,
such derived rings never have finite global dimension,
unless they are equivalent to an ordinary regular local ring.
 
There is another homological approach to defining regular local rings by using the notion of a regular sequence.
A noetherian local ring $(A,\m)$ is regular if and only if the maximal ideal $\m$ can be generated by an $A$-regular sequence.
In that case, such a sequence must have length $d = \dim(A)$,
and every sequence of elements of $A$ of length $d$ which generates $\m$ is a regular sequence.
This is the approach we will take in this paper. 

We will study non-positively graded (in cohomological grading) commutative noetherian DG-rings.
All facts and definitions used in this paper about commutative DG-rings will be recalled in Section \ref{sec:prem} below.
For a commutative noetherian DG-ring $A$, 
associated to it is an ordinary commutative noetherian ring $\mrm{H}^0(A)$.
When this ring is local with a maximal ideal $\bar{\m}$,
we say that $(A,\bar{\m})$ is a noetherian local DG-ring.
For noetherian local DG-rings $(A,\bar{\m})$ with bounded cohomology, 
the notion of an $A$-regular sequence in $\bar{\m}$ was introduced and studied,
first by Minamoto in \cite{Mi2},
and later by the author in \cite{ShCM, ShKos}.
It is recalled in Section \ref{sec:reg-seq} below.
Once one has the notion of an $A$-regular sequence,
one can make the following definition:
\begin{dfn*}
Let $(A,\bar{\m})$ be a commutative noetherian local DG-ring with bounded cohomology.
We say that $A$ is a sequence-regular DG-ring if the maximal ideal $\bar{\m}$ can be generated by an $A$-regular sequence.
\end{dfn*}

With this definition in hand, 
we wish to study this notion.
A natural question that immediately arises is whether this notion is stable under localization.
In the classical context of ordinary commutative noetherian rings,
this was known in the 1950s as the localization problem,
and at the time it was a major open problem in commutative algebra.
This was the case until Serre in \cite{Serre}, 
and independently Auslander and Buchsbaum in \cite{AusBuch}
showed that the usual definition of a regular local ring is equivalent to it being of finite global dimension.
With this introduction of homological methods to commutative algebra,
the localization problem was solved.

There is a notion of localization of commutative noetherian DG-rings $A$,
with respect to prime ideals $\bar{\p} \in \opn{Spec}(\mrm{H}^0(A))$.
It is denoted by $A_{\bar{\p}}$ and recalled in Section \ref{sec:loc} below.
In Corollary \ref{cor:Serre} we prove:

\begin{cthm}\label{cthm:local}
Let $(A,\bar{\m})$ be a sequence-regular noetherian local DG-ring.
Then for any prime $\bar{\p} \in \opn{Spec}(\mrm{H}^0(A))$,
the localization $A_{\bar{\p}}$ is also a sequence-regular DG-ring.
\end{cthm}

In order to prove this, we are unable to proceed as in the classical case.
The reason for that is, as mentioned above,
that the characterization of regular local rings as the rings of finite global dimension is simply false in the DG-setting.
Any non-trivial DG-ring with bounded cohomology has infinite global dimension.
Instead, we are able to provide another characterization of sequence-regular DG-rings,
one which is also easily seen to be stable under localization.

The key ingredient to obtain such a characterization is to use Cohen–Macaulay DG methods.
The notion of a Cohen–Macaulay DG-ring was first described in \cite{ShCM},
and was later extensively developed in \cite{ShKos}.
In this paper we will continue the development of this notion,
and use the theory of Cohen–Macaulay DG-rings to prove the following characterization of sequence-regualr DG-rings:
\begin{cthm}\label{cthm:charac}
Let $(A,\bar{\m})$ be a commutative noetherian local DG-ring with bounded cohomology.
Then the following are equivalent:
\begin{enumerate}
\item The DG-ring $(A,\bar{\m})$ is sequence-regular.
\item The DG-ring $(A,\bar{\m})$ is a Cohen–Macaulay DG-ring, and moreover, the ring $\mrm{H}^0(A)$ is a regular local ring.
\end{enumerate}
\end{cthm}  
This result is repeated as Theorem \ref{thm:main} in the body of the paper.
See Remark \ref{rem:both-needed} for a discussion about the necessity of all conditions in this result.
Since Cohen–Macaulay DG-rings and regular local rings are stable under localization,
Theorem \ref{cthm:local} follows now immediately from this result.
We view this result as the \textbf{most important contribution of this paper},
because the definition of sequence regularity is a natural generalization,
but without Theorem \ref{cthm:charac} it seems very difficult in practice to show that a given DG-ring is sequence-regular.

Once one knows Theorem \ref{cthm:local} holds,
we may define global sequence-regular DG-rings as those commutative noetherian DG-rings such that all of their localizations are sequence-regular. 
Then, a characterization of these global sequence-regular DG-rings,
similar to Theorem \ref{cthm:charac} easily follows.

Let us now describe the contents of this paper.
As we already mentioned, Section \ref{sec:prem} contains the various preliminaries used in this paper.
In particular, we recall in it the notion of a Cohen–Macaulay DG-ring,
but also the more primitive notion of a local-Cohen–Macaulay DG-ring.
A Cohen–Macaulay DG-ring is defined to be a DG-ring such that all of its localizations are local-Cohen–Macaulay,
and in general the notion of a local-Cohen–Macaulay DG-ring is not stable under localization.

Section \ref{sec:lcm} is another chapter in the theory of local-Cohen–Macaulay and Cohen–Macaulay DG-ring.
Its main result, Theorem \ref{thm:lCMisCM} gives a sufficient condition for 
local-Cohen–Macaulay DG-rings to be Cohen–Macaulay.
This result is crucial for the proof of Theorem \ref{cthm:charac}.

In Section \ref{sec:seq-reg}, 
after proving some further results about Cohen–Macaulay DG-rings and regular sequences,
we define sequence-regular local DG-rings and prove Theorem \ref{cthm:charac}.

After that, in Section \ref{sec:operations} we study the behavior of the sequence-regular property with respect to various natural operations.
In particular, we show that the sequence-regularity condition is stable under localization and derived completion,
and investigate its transfer along flat local maps.
We further define global sequence-regular DG-rings in this setting,
and also give in Theorem \ref{thm:seq-reg} a necessary and sufficient condition for the derived quotient of a sequence-regular DG-ring to be sequence-regular.

The topic of Section \ref{sec:resField} is a generalization of the notion of a residue field of a local ring to the DG setting.
If $(A,\m)$ is a noetherian local ring, 
its residue field $\kappa(A) = A/\m$ is a simple basic construction in commutative algebra which is very useful.
This construction is not available in the setting of commutative DG-rings,
but in this section we show that for a sequence-regular local DG-ring $(A,\bar{\m})$,
one may associate a canonical DG-algebra over $A$, denoted by $\kappa(A)$
which generalizes the classical construction of the residue field.
We call $\kappa(A)$ the residue DG-field associated to $A$,
and dedicate Section \ref{sec:resField} to the study of its properties,
showing that various properties of the classical residue field extend to the residue DG-field.
In particular, we prove in Theorem \ref{thm:charac-residue} that the existence of a residue DG-field which 
satisfies certain natural properties is actually equivalent to the DG-ring being sequence-regular.
Thus, those residue DG-fields are available \textbf{only} in the sequence-regular case.
Then, more generally, for a (not necessarily local) sequence-regular DG-ring $A$,
and for any $\bar{\p} \in \opn{Spec}(\mrm{H}^0(A))$,
we construct the canonical DG-algebra $\kappa(A,\bar{\p})$ which is an analogue of the classical $\kappa(\p) = A_{\p}/\p A_{\p}$;
the residue field of a ring $A$ at a prime ideal $\p$.

The topic of Section \ref{sec:generic} is the ubiquity of sequence-regular DG-rings.
In one of its main results, Corollary \ref{cor:geometric} we show:

\begin{cthm}
Let $\K$ be a perfect field, and let $X$ be an eventually coconnective noetherian DG-scheme over $\K$.
Suppose that the classical scheme underlying $X$ is an algebraic variety over $\K$.
Then $X$ is generically sequence-regular.
\end{cthm}

This result means that whenever one studies the derived algebraic geometry in a bounded setting,
if the underlying scheme is a classical algebraic variety over a perfect field,
then the DG-scheme one works with is sequence-regular on a dense open set.
This generalizes the classical fact that an algebraic variety over a perfect field is nonsingular on a dense open set.

In the final Section \ref{sec:special} we return to the classical sequence of containments of special rings
\[
\mbox{Regular rings} \subsetneq \mbox{Lci rings} \subsetneq \mbox{Gorenstein rings} \subsetneq \mbox{Cohen–Macaulay rings} 
\]
and discuss how this picture generalizes to the DG-setting.
Here, regular rings are generalized by sequence-regular DG-rings,
local complete intersection (lci) rings are generalized by quasi-smooth DG-rings,
and Gorenstein and Cohen–Macaulay rings are generalized by Gorenstein and Cohen–Macaulay DG-rings.
In this section we discuss to what extent these containments carry (and fail to carry) over to the DG-setting.

We finish the introduction by mentioning that an idea similar to that of sequence-regularity has appeared
before in the papers \cite{BGS,Greenlees,GHS} where the notion of an s-regular ring spectra was studied.
However, it seems that because of the fact that commutative rings are much more similar to commutative DG-rings then to ring spectra,
the impression of the author is that the resulting theory is not as similar to the theory of regular local rings as the theory developed in this paper.

\section{Preliminaries}\label{sec:prem}

In this section we gather some preliminaries that will be used throughout this paper.

\subsection{Commutative DG-rings}
The main objects of study in this paper are commutative differential graded (always abbreviated as DG) rings.
These are defined to be $\mathbb{Z}$-graded algebras $A = \bigoplus_{n=-\infty}^\infty A^n$,
equipped with a differential $d:A\to A^{n+1}$ which satisfy a Leibnitz rule 
$d(a\cdot b) = d(a)\cdot b + (-1)^{\deg(a)}\cdot a \cdot d(b)$. 
We will further assume that all DG-rings in this paper are non-positively graded, which means that $A^n = 0 $ for all $n>0$.
Such DG-rings will be called non-positive DG-rings.

The homotopy category of DG-rings is obtained from the category of DG-rings by formally inverting all quasi-isomorphisms.
It can be realized as the homotopy category of a Quillen model category,
since the category of non-positive DG-rings has such a model structure in which the weak equivalences are exactly the quasi-isomorphisms.

A DG-ring $A$ will be called commutative if $b\cdot a = (-1)^{\deg(a)\cdot \deg(b)}\cdot a \cdot b$,
and $a^2 = 0$ if $a$ is an homogenous element of odd degree.

In addition to being non-positive, we will further assume that all DG-rings in this paper are commutative and unital.
In particular, all rings appearing in this paper (which are just DG-rings concentrated in degree $0$) will be commutative and unital.
For a complete account about DG-rings and their derived categories, we refer the reader to \cite{YeBook}.

The derived category of a DG-ring $A$ is the category obtained from inverting quasi-isomorphisms in the category of DG-modules over $A$.
This is a triangulated category which we will denote by $\cat{D}(A)$.
The full triangulated subcategory which consists of DG-modules whose cohomology is bounded below (respectively above)
will be denoted by $\cat{D}^{+}(A)$ (resp. $\cat{D}^{-}(A)$).
The full triangulated subcategory which consists of DG-modules with bounded cohomology will be denoted by $\cat{D}^{\mrm{b}}(A)$.
The infimum, the supremum and amplitude of a DG-module $M$ are defined as
\[
\inf(M)=\inf\{n \in \mathbb{Z} \mid \mrm{H}^n(M) \ne 0\}
\quad
\sup(M)=\sup\{n \in \mathbb{Z} \mid \mrm{H}^n(M) \ne 0\}
\]
and $\amp(M) := \sup(M) - \inf(M)$. 
In particular, we see that $M \in \cat{D}^{\mrm{b}}(A)$ if and only if $\amp(M) < \infty$.

Given a commutative non-positive DG-ring $A$,
the cohomology $\mrm{H}^0(A)$ is a commutative ring,
and there is a natural map $\pi_A:A \to \mrm{H}^0(A)$.
We will sometimes restrict to its degree $0$ part which is simply the surjection of commutative rings $\pi^0_A:A^0 \to \mrm{H}^0(A)$.

Since we will frequently work with both $A$ and $\mrm{H}^0(A)$ at the same time,
we will follow the convention of \cite{YeDual}, and write $\bar{?}$ to denote elements and ideals of $\mrm{H}^0(A)$.
Thus for example, a typical element of $\mrm{H}^0(A)$ will be denoted by $\bar{x}$,
and similarly, a typical prime ideal of $\mrm{H}^0(A)$ will be denoted by $\bar{\p}$.

For any $M \in \cat{D}(A)$ and any $n\in \mathbb{Z}$,
the cohomology $\mrm{H}^n(M)$ has the structure of a module over the ring $\mrm{H}^0(A)$.
We will say that a DG-ring $A$ is noetherian if $\mrm{H}^0(A)$ is a noetherian ring,
and for all $i<0$, the $\mrm{H}^0(A)$-module $\mrm{H}^i(A)$ is finitely generated.
Almost all DG-rings in this paper will be noetherian,
and we will often also assume that they have bounded cohomology.
The latter means that $\mrm{H}^i(A) = 0$ for all $i\ll 0$.
We will further say that a noetherian DG-ring $A$ has a noetherian model if there is some noetherian ring $B$ and a map of DG-rings $B\to A$ such that the induced map $B\to\mrm{H}^0(A)$ is surjective.

If $A$ is a commutative noetherian DG-ring,
and if $M \in \cat{D}(A)$, 
we will say that $M$ has finitely generated cohomology over $A$ if for all $n\in \mathbb{Z}$,
the $\mrm{H}^0(A)$-modules $\mrm{H}^n(M)$ are all finitely generated.
We will denote by $\cat{D}_{\mrm{f}}(A)$ the full triangulated subcategory which consists of all DG-modules with finitely generated cohomology. 
We will further write $\cat{D}^{\mrm{b}}_{\mrm{f}}(A)$ for the full triangulated subcategory which consists of DG-modules which have bounded cohomology which is finitely generated. In other words, 
\[
\cat{D}^{\mrm{b}}_{\mrm{f}}(A) = \cat{D}_{\mrm{f}}(A) \cap \cat{D}^{\mrm{b}}(A).
\]

\subsection{Localization}\label{sec:loc}
A commutative noetherian DG-ring $A$ will be called a local DG-ring if the noetherian ring $\mrm{H}^0(A)$ is a local ring.
In that case, if $\bar{\m}$ is the maximal ideal of the local ring $\mrm{H}^0(A)$,
we will indicate this by simply saying that $(A,\bar{\m})$ is a noetherian local DG-ring.

One way to obtain noetherian local DG-rings is using the DG version of localization,
introduced in \cite[Section 4]{YeDual}.
It is defined as follows. 
If $A$ is a commutative DG-ring, and if $\bar{\p} \in \opn{Spec}(\mrm{H}^0(A))$,
let $\p = (\pi^0_A)^{-1}(\bar{\p}) \in \opn{Spec}(A^0)$,
and define $A_{\bar{\p}} = A\otimes_{A^0} A^0_{\p}$.
Then one has $\mrm{H}^0(A_{\bar{\p}}) = \mrm{H}^0(A)_{\bar{\p}}$.
In particular, if $A$ is a noetherian DG-ring,
then $A_{\bar{\p}}$ is a noetherian local DG-ring.
One then obtains a localization functor for $\cat{D}(A) \to \cat{D}(A_{\bar{\p}})$ which
sends a DG-module $M$ to $M_{\bar{\p}} = M\otimes^{\mrm{L}}_A A_{\bar{\p}}$.

\subsection{Regular sequences and Koszul complexes}\label{sec:reg-seq}

Let $(A,\bar{\m})$ be a noetherian local DG-ring,
and let $M\in \cat{D}^{+}(A)$.
An element $\bar{x} \in \bar{\m}$ is called an $M$-regular element if it an $\mrm{H}^{\inf(M)}(M)$-regular element.
In other words, $\bar{x}$ is $M$-regular if multiplication by it on the bottom cohomology of $M$ is an injective map.

To define regular sequences of longer length, 
we first recall the notion of a Koszul complex in this context.

Given any commutative DG-ring $A$,
any given $\bar{a}_1,\dots,\bar{a}_n \in \mrm{H}^0(A)$,
the Koszul complex $K(A;\bar{a}_1,\dots,\bar{a}_n)$ is a commutative non-positive DG-ring,
equipped with a map $A\to K(A;\bar{a}_1,\dots,\bar{a}_n)$.
One way to define it is as follows:
for each $1\le i \le n$, choose some $a_i \in A^0$ such that $\pi_A(a_i) = \bar{a}_i$.
Then, consider the Koszul complex $K(A^0;a_1,\dots,a_n)$ over the ring $A^0$,
and finally, let
\[
K(A;\bar{a}_1,\dots,\bar{a}_n) = K(A^0;a_1,\dots,a_n)\otimes_{A^0} A.
\]
As explained in \cite[Proposition 2.6]{ShKos} (or \cite[Lemma 2.8]{Mi2} in case $n=1$),
the result is independent of the chosen lifts of $\bar{a}_1,\dots,\bar{a}_n$,
up to isomorphism in the homotopy category of DG-rings.
As explained in the introduction of \cite{ShKos},
one should think of this construction as the derived quotient of $A$ with respect to $\bar{a}_1,\dots,\bar{a}_n$.

Returning to the topic of regular sequences,
if $(A,\bar{\m})$ is a noetherian local DG-ring,
and if $M \in \cat{D}^{+}(A)$,
a sequence $\bar{a}_1,\dots,\bar{a}_n \in \bar{\m}$ will be called an $M$-regular sequence if
$\bar{a}_1$ is $M$-regular,
and the sequence $\bar{a}_2,\dots,\bar{a}_n$ is $K(M;\bar{a}_1)$-regular.

Given an ideal $\bar{I} \subseteq \mrm{H}^0(A)$,
and assuming that $M \in \cat{D}^{+}_{\mrm{f}}(A)$,
the $\bar{I}$-depth of $M$ is defined to be the number $\opn{depth}_A(\bar{I},M) = \inf(\mrm{R}\opn{Hom}_A(\mrm{H}^0(A)/\bar{I},M))$.
When $\bar{I} = \bar{\m}$, it is simply called the depth of $M$, 
and it is denoted by $\opn{depth}_A(M)$.
The maximal length of an $M$-regular sequence contained in $\bar{I}$ is a well defined integer 
denoted by $\opn{seq.depth}_A(\bar{I},M)$ and called the $\bar{I}$ sequential depth of $M$.
Again, when $\bar{I}=\bar{\m}$, we will denote it by $\opn{seq.depth}_A(M)$.

Following \cite[Definition 5.8]{ShCM},
a prime ideal $\bar{\p} \in \opn{Spec}(\mrm{H}^0(A))$ is called an associated prime ideal of $M$ if
$\opn{depth}_{A_{\bar{\p}}}(M_{\bar{\p}}) = \inf(M_{\bar{\p}})$.
The set of all associated primes of $M$ is denoted by $\opn{Ass}_A(M)$.
We shall also require the dual notion of the set $W^A_0(M)$,
also introduced in \cite[Definition 5.8]{ShCM}.
This is defined to be the set of all $\bar{\p} \in \opn{Spec}(\mrm{H}^0(A))$,
such that there is an equality
\[
\sup(M_{\bar{\p}}) = \sup_{n\in \mathbb{Z}}(\dim(\mrm{H}^n(M))+n),
\]
where here $\dim$ denotes the Krull dimension of the $\mrm{H}^0(A)$-modules $\mrm{H}^n(M)$.

The notion of a regular sequences recalled here is similar to the notion of an $M$-regular sequence where $M$ is a cochain complex over a noetherian ring. These were extensively studied by Christensen in \cite{ChrI,ChrII}.
We further mention that another notion of regularity in a DG setting was studied by Iyengar in \cite{Iyen}. 

\subsection{Dualizing DG-modules}

If $A$ is a commutative noetherian DG-ring,
a dualizing DG-module over $A$ is a DG-module $R \in \cat{D}^{+}_{\mrm{f}}(A)$,
such that $R$ has finite injective dimension over $A$ (as in \cite[Definition 12.4.8(2)]{YeBook}),
and moreover, the natural map $A \to \mrm{R}\opn{Hom}_A(R,R)$ is an isomorphism.
This notion generalizes the notion of a dualizing complex which is due to Grothendieck.
Dualizing DG-modules were studied in \cite{FIJ, YeDual}.

By \cite[Theorem 4.1]{ShCM},
if $R$ is a dualizing DG-module over $A$,
there is always an inequality $\amp(R) \ge \amp(A)$.

If $(A,\bar{\m})$ is a noetherian local DG-ring,
and if $R,S$ are two different dualizing DG-modules over $A$,
then by \cite[Theorem III]{FIJ} and \cite[Corollary 7.16]{YeDual}
there exist some $n \in \mathbb{Z}$ such that $R \cong S[n]$.
Here $[n]$ denotes the shift functor of the triangulated category $\cat{D}(A)$.

\subsection{Cohen–Macaulay DG-rings}

If $(A,\bar{\m})$ is a commutative noetherian local DG-ring with bounded cohomology,
then by \cite[Corollary 5.5]{ShCM},
there is an inequality 
\[
\opn{seq.depth}_A(A) \le \dim(\mrm{H}^0(A)).
\]
If this inequality is an equality,
that is, if there exist an $A$-regular sequence in $\bar{\m}$ which is of length $\dim(\mrm{H}^0(A))$,
then the DG-ring $A$ is called a local-Cohen–Macaulay DG-ring.

The reason for this somewhat cumbersome terminology is that,
as demonstrated in \cite[Example 8.3]{ShCM},
the classical fact that the localization of a Cohen–Macaulay ring is again Cohen–Macaulay does not always generalize to the DG setting.
Thus, there exist examples of local-Cohen–Macaulay DG-rings $A$,
and prime ideals $\bar{\p} \in \opn{Spec}(\mrm{H}^0(A))$,
such that $A_{\bar{\p}}$ is not local-Cohen–Macaulay.

Because of this, we explicitly distinguish between local-Cohen–Macaulay DG-rings
and Cohen–Macaulay DG-ring. 
The latter are, by definition commutative noetherian DG-rings $A$,
with bounded cohomology, such that for all $\bar{\p} \in \opn{Spec}(\mrm{H}^0(A))$,
the local DG-ring $A_{\bar{\p}}$ is local-Cohen–Macaulay.

If $(A,\bar{\m})$ is a noetherian local DG-ring with bounded cohomology,
and if $A$ admits a dualizing DG-module $R$,
then $A$ is local-Cohen–Macaulay if and only if $\amp(A) = \amp(R)$.
One may also define local-Cohen–Macaulay using local cohomology,
but this will not be needed in this paper.

\section{Local-Cohen–Macaulay DG-rings with constant amplitude}\label{sec:lcm}

The aim of this section is to show that under very mild hypothesis (namely, that the underlying space is catenary),
any local-Cohen–Macaulay DG-ring with constant amplitude is Cohen–Macaulay. 
Recall, following \cite[Definition 3.5]{ShKos},
that a commutative noetherian DG-ring $A$ with bounded cohomology is called a DG-ring with constant amplitude,
if $\opn{Supp}(\mrm{H}^{\inf(A)}(A)) = \opn{Spec}(\mrm{H}^0(A))$.
The reason for this terminology is that this is equivalent to the fact that $\amp(A_{\bar{\p}}) = \amp(A)$ for all $\bar{\p} \in \opn{Spec}(\mrm{H}^0(A))$.

Local-Cohen–Macaulay DG-rings often have constant amplitude:
\begin{prop}\label{prop:const}
Let $(A,\bar{\m})$ be a local-Cohen–Macaulay DG-ring,
and suppose that $\mrm{H}^0(A)$ has a unique minimal prime ideal.
Then $A$ has constant amplitude.
\end{prop}
\begin{proof}
Let $n = \inf(A)$. 
According to \cite[Proposition 4.11]{ShCM},
the fact that $A$ is local-Cohen–Macaulay implies that
\[
\dim(\mrm{H}^n(A)) = \dim(\mrm{H}^0(A)).
\]
Then, by \cite[Proposition 8.5]{ShCM},
the fact that $\mrm{H}^0(A)$ has a unique minimal prime ideal implies that
\[
\opn{Supp}(\mrm{H}^n(A)) = \opn{Spec}(\mrm{H}^0(A)),
\]
so that $A$ has constant amplitude.
\end{proof}

Here is a useful property of DG-rings with constant amplitude.

\begin{prop}\label{prop:ass-primes}
Let $(A,\bar{\m})$ be a local noetherian DG-ring with bounded cohomology and constant amplitude,
Then there is an equality $\opn{Ass}(A) = \opn{Ass}(\mrm{H}^{\inf(A)}(A))$.
\end{prop}
\begin{proof}
Given $\bar{\p} \in \opn{Spec}(\mrm{H}^0(A))$,
by definition we have that $\bar{\p} \in \opn{Ass}(A)$ if 
$\opn{depth}_{A_{\bar{\p}}}(A_{\bar{\p}}) = \inf(A_{\bar{\p}})$.
According to \cite[Proposition 3.3]{ShCM},
this is the case if and only if $\bar{\p}\cdot \mrm{H}^0(A_{\bar{\p}})$
is an associated prime of $\mrm{H}^{\inf(A_{\bar{\p}})}(A_{\bar{\p}})$.
The fact that $A$ has constant amplitude implies that $\inf(A_{\bar{\p}}) = \inf(A)$,
so we see that $\bar{\p} \in \opn{Ass}(A)$ if and only if $\bar{\p}\cdot \mrm{H}^0(A_{\bar{\p}})$
is an associated prime of
\[
\mrm{H}^{\inf(A)}(A_{\bar{\p}}) \cong \left(\mrm{H}^{\inf(A)}(A)\right)_{\bar{\p}}.
\]
According to \cite[Theorem 6.2]{Mat},
this is equivalent to $\bar{\p} \in \opn{Ass}(\mrm{H}^{\inf(A)}(A))$,
as claimed.
\end{proof}

The following result is similar to \cite[Proposition 3.9]{ShKos}. 
The conclusion is the same, but the assumptions we make here are a bit different (there $A$ is assumed to be Cohen–Macaulay,
and not just local-Cohen–Macaulay, but on the other hand, it is not assumed there that $A$ has constant amplitude).

\begin{prop}\label{prop:minimalPrimes}
Let $(A,\bar{\m})$ be a local-Cohen–Macaulay DG-ring with constant amplitude.
Then $\bar{\p} \in \opn{Ass}(A)$ if and only if $\bar{\p} \in W^A_0(A)$ if and only if $\opn{ht}(\bar{\p}) = 0$.
\end{prop}
\begin{proof}
Suppose first that $\bar{\p} \in \opn{Ass}(A)$.
According to \cite[Theorem 3.8]{ShKos},
we know that 
\[
\opn{seq.depth}_A(A) \le \dim(\mrm{H}^0(A/\bar{\p})),
\]
but $A$ being local-Cohen–Macaulay means that $\opn{seq.depth}_A(A) = \dim(\mrm{H}^0(A))$,
so we must have that $\dim(\mrm{H}^0(A/\bar{\p})) = \dim(\mrm{H}^0(A))$,
which implies that $\opn{ht}(\bar{\p}) = 0$.
Conversely, if $\opn{ht}(\bar{\p}) = 0$,
then $\bar{\p}$ is a minimal prime ideal of $\opn{Spec}(\mrm{H}^0(A))$,
and since $A$ has constant amplitude,
it is also a minimal prime ideal of $\opn{Supp}(\mrm{H}^{\inf(A)}(A))$.
By \cite[Theorem 6.5]{Mat} this implies that $\bar{\p} \in \opn{Ass}(\mrm{H}^{\inf(A)}(A))$,
so Proposition \ref{prop:ass-primes} implies that $\bar{\p} \in \opn{Ass}(A)$.
Finally, the fact that $\bar{\p} \in W^A_0(A)$ if and only if $\opn{ht}(\bar{\p}) = 0$
is proved in \cite[Proposition 3.9]{ShKos},
and the same proof given there applies in this setting as well.
\end{proof}

The next lemma is a result in dimension theory of rings which will be needed in the result following it.
\begin{lem}\label{lem:caten}
Let $(A,\m)$ be a noetherian local ring, and suppose that $A$ is catenary.
Let $\p \in \opn{Spec}(A)$, and let $x \in \p$ be an element such that $\dim(A/x) = \dim(A) - 1$.
Then $\dim((A/x)_{\p}) = \dim(A_{\p})-1$.
\end{lem}
\begin{proof}
By Krull's Hauptidealsatz we know that $\dim(A_{\p})-1 \le \dim((A/x)_{\p}) \le \dim(A_{\p})$.
Suppose that $\dim((A/x)_{\p}) = \dim(A_{\p}) = n$.
Then there is a maximal chain of prime ideals contained in $\p$ in $A$ of the form
\[
\p_0 \subsetneq \p_1 \subsetneq \dots \subsetneq \p_n = \p,
\]
and moreover, there is such a chain of length $n$ with the property that $x \in \p_0$,
and that $\p_0$ has minimal height among all the prime ideals that contain $x$.
The fact that $A$ is catenary implies that we may extend this chain to a maximal chain
\[
\p_0 \subsetneq \p_1 \subsetneq \dots \subsetneq \p_n = \p \subsetneq \dots \subsetneq \p_m = \m
\]
whose length is equal to $m = \dim(A/\p_0) = \dim(A/x) = \dim(A)-1$.
However, the maximality of this chain also implies that $\dim(A/\p) = m-n$,
so we deduce that $\dim(A) = \dim(A_{\p}) + \dim(A/\p) = m = \dim(A)-1$,
which is a contradiction. Hence, $\dim((A/x)_{\p}) = \dim(A_{\p})-1$.
\end{proof}

The next result is essentially a generalization of \cite[Theorem 8.4]{ShCM}.
The proof structure of this result is built on the proof of \cite[Theorem 17.3(iii)]{Mat},
but the situation here is considerably more complicated.

\begin{thm}\label{thm:lCMisCM}
Let $(A,\bar{\m})$ be a local-Cohen–Macaulay DG-ring with constant amplitude such that $\mrm{H}^0(A)$ is catenary.
Then $A$ is a Cohen–Macaulay DG-ring.
\end{thm}
\begin{proof}
Let $\bar{\p} \in \opn{Spec}(\mrm{H}^0(A))$.
We must show that $A_{\bar{\p}}$ is a local-Cohen–Macaulay DG-ring.
We will prove this by induction on $\opn{seq.depth}(\bar{\p},A)$.
We remark that this quantity is finite by \cite[Proposition 3.2]{ShKos}.
Assume first that 
\[
\opn{seq.depth}(\bar{\p},A) = 0.
\]
Then every element of $\bar{\p}$ is a zero-divisor of $\mrm{H}^{\inf(A)}(A)$.
By \cite[Thoerem 6.1]{Mat},
this implies that $\bar{\p}$ is contained in the union of the associated primes of $\mrm{H}^{\inf(A)}(A)$,
and since there are only finitely many such primes, 
by the prime avoidance lemma, it must be contained in one of these associated primes, say
\[
\bar{\p} \subseteq \bar{\q} \in \opn{Ass}(\mrm{H}^{\inf(A)}(A)).
\]
According to Proposition \ref{prop:ass-primes}, 
we see that $\bar{\q} \in \opn{Ass}(A)$,
so by Proposition \ref{prop:minimalPrimes} we deduce that $\opn{ht}(\bar{\q}) = 0$,
which implies that $\bar{\p} = \bar{\q}$.
This shows that $\dim(\mrm{H}^0(A_{\bar{\p}})) = 0$.
According to \cite[Proposition 4.8]{ShCM},
this implies that $A_{\bar{\p}}$ is local-Cohen–Macaulay.
Suppose now that 
\[
\opn{seq.depth}(\bar{\p},A) = n > 0,
\]
and assume by induction that for any local-Cohen–Macaulay DG-ring $B$ with constant amplitude,
and any $\bar{\q} \in \opn{Spec}(\mrm{H}^0(B))$ such that $\opn{seq.depth}(\bar{\q},B) < n$,
it holds that $B_{\bar{\q}}$ is local-Cohen–Macaulay.
If $\opn{ht}(\bar{\p}) = 0$ (which actually cannot happen),
the result follows as above, 
so we may assume that $\opn{ht}(\bar{\p}) > 0$.
Then $\bar{\p}$ is not a minimal prime of $\opn{Spec}(\mrm{H}^0(A))$,
and by Proposition \ref{prop:minimalPrimes},
this implies that $\bar{\p}$ is not contained in the union of the associated primes of $A$,
and also $\bar{\p}$ is not contained in the union of the prime ideals in $W^A_0(A)$.
By the prime avoidance lemma, we may find $\bar{x} \in \bar{\p}$
such that for all minimal primes $\bar{\q}$ of $\opn{Spec}(\mrm{H}^0(A))$,
it holds that $\bar{x} \notin \bar{\q}$. 
According to \cite[Proposition 5.13]{ShCM},
this implies that $\bar{x}$ is $A$-regular,
while by \cite[Proposition 5.16]{ShCM},
it holds that 
\[
\dim(\mrm{H}^0(A)/(\bar{x})) = \dim(\mrm{H}^0(A)) - 1.
\]
Let us denote by $B$ the DG-ring $K(A;\bar{x})$.
Since $A$ is local-Cohen–Macaulay, we deduce that
\[
\opn{seq.depth}(K(A;\bar{x})) = \opn{seq.depth}(A) - 1 = \dim(\mrm{H}^0(A))-1 = \dim(\mrm{H}^0(K(A;\bar{x})),
\]
which implies that $B$ is local-Cohen–Macaulay.
By \cite[Lemma 3.13]{ShKos}, we further know that $B$ has constant amplitude.
Let $\bar{\q}$ be the image of $\bar{\p}$ in $\mrm{H}^0(A/\bar{x})$.
Then by definition, we have that
\[
\opn{seq.depth}(\bar{\q},B) = \opn{seq.depth}(\bar{\p},A) - 1 = n-1.
\]
Hence, the induction hypothesis implies that $B_{\bar{\q}}$ is local-Cohen–Macaulay.
By definition, this implies that
\[
\opn{seq.depth}(\bar{\q},B_{\bar{\q}}) = \dim(\mrm{H}^0(B_{\bar{\q}})).
\]
According to \cite[Proposition 2.11]{ShKos}, 
there is a DG-ring isomorphism
\[
B_{\bar{\q}} \cong K(A_{\bar{\p}},\frac{(\bar{x})}{1}),
\]
where $\frac{(\bar{x})}{1}$ denotes the image of $\bar{x}$ in the localization $\mrm{H}^0(A_{\bar{\p}})$.
As $\bar{x}$ is $A$-regular, we know that $\bar{x}$ is $\mrm{H}^{\inf(A)}(A)$-regular,
and since $\bar{x} \in \bar{\p}$,
we deduce that $\frac{(\bar{x})}{1}$ is $\mrm{H}^{\inf(A)}(A)_{\bar{\p}}$-regular,
so by the constant amplitude assumption, we deduce that $\frac{(\bar{x})}{1}$ is $A_{\bar{\p}}$-regular.
Hence, we obtain that
\begin{gather*}
\opn{seq.depth}(\bar{\p},A_{\bar{\p}}) = \opn{seq.depth}\left(\bar{\p},K(A_{\bar{\p}},\frac{(\bar{x})}{1})\right) + 1 =\\ \opn{seq.depth}(\bar{\q},B_{\bar{\q}}) + 1 = \dim(\mrm{H}^0(B_{\bar{\q}})) + 1 = \dim(\mrm{H}^0(A_{\bar{\p}}))
\end{gather*}
where the last equality follows from Lemma \ref{lem:caten} and the fact that $\mrm{H}^0(B) = \mrm{H}^0(A)/\bar{x}$.
This shows that $A_{\bar{\p}}$ is local-Cohen–Macaulay, as claimed. 
\end{proof}

\section{Sequence-regular DG-rings}\label{sec:seq-reg}

In this section we arrive to the main object of study of this paper,
sequence-regular DG-rings. 
The next result was contained in the proof of \cite[Theorem 4.2]{ShKos}.
We state it as a separate result because of its importance in what follows.

\begin{prop}\label{prop:kosAmp}
Let $(A,\bar{\m})$ be a Cohen–Macaulay local DG-ring with constant amplitude,
let $\bar{x}_1,\dots,\bar{x}_n \in \bar{\m}$,
and denote by $\bar{I}$ the ideal in $\mrm{H}^0(A)$ generated by $\bar{x}_1,\dots,\bar{x}_n$.
Then there is an equality
\[
\amp\left(K(A;\bar{x}_1,\dots,\bar{x}_n)\right) = n - \dim(\mrm{H}^0(A)) + \dim(\mrm{H}^0(A)/\bar{I}) + \amp(A).
\]
\end{prop}
\begin{proof}
Because $K(A;\bar{x}_1,\dots,\bar{x}_n)$ is a non-positive DG-ring, 
we know that
\[
\amp\left(K(A;\bar{x}_1,\dots,\bar{x}_n)\right) = -\inf\left(K(A;\bar{x}_1,\dots,\bar{x}_n)\right).
\]
By \cite[Proposition 3.17]{ShKos}, there is an equality
\[
\inf\left(K(A;\bar{x}_1,\dots,\bar{x}_n)\right) = \opn{depth}(\bar{I},A) - n.
\]
According to \cite[Proposition 3.2]{ShKos}, we know that
\[
\opn{depth}(\bar{I},A) = \opn{seq.depth}_A(\bar{I},A) + \inf(A).
\]
Finally, \cite[Theorem 3.16]{ShKos} says that
\[
\opn{seq.depth}_A(\bar{I},A) = \dim(\mrm{H}^0(A)) - \dim(\mrm{H}^0(A)/\bar{I}).
\]
Combining all these equalities together establishes the result.
\end{proof}

The following result is a generalization of \cite[Theorem 2.1.2(c)]{BH} to the DG setting.

\begin{prop}\label{prop:sopIsRegular}
Let $(A,\bar{\m})$ be a Cohen–Macaulay local DG-ring with constant amplitude,
and let $\bar{x}_1,\dots,\bar{x}_n\in \bar{\m}$ be elements of $\mrm{H}^0(A)$.
Then $\bar{x}_1,\dots,\bar{x}_n$ is an $A$-regular sequence if and only if 
\[
\dim(\mrm{H}^0(A)/(\bar{x}_1,\dots, \bar{x}_n)) = \dim(\mrm{H}^0(A)) - n.
\]
\end{prop}
\begin{proof}
By \cite[Lemma 2.13(2)]{Mi2}, 
it holds that $\bar{x}_1,\dots,\bar{x}_n$ is an $A$-regular sequence if and only if
\[
\inf(A) = \inf(K(A;\bar{x}_1,\dots,\bar{x}_n)),
\]
or equivalently (since both of these DG-rings are non-positive),
\[
\amp(A) = \amp(K(A;\bar{x}_1,\dots,\bar{x}_n)).
\]
To compute the latter, we invoke Proposition \ref{prop:kosAmp}.
Denoting bu $\bar{I}$ the ideal in $\mrm{H}^0(A)$ generated by $\bar{x}_1,\dots,\bar{x}_n$,
there is an equality
\[
\amp(K(A;\bar{x}_1,\dots,\bar{x}_n) = n - \dim(\mrm{H}^0(A)) + \dim(\mrm{H}^0(A)/\bar{I}) + \amp(A).
\]
It follows that $\bar{x}_1,\dots,\bar{x}_n$ is an $A$-regular sequence if and only if 
\[
\dim(\mrm{H}^0(A)/\bar{I}) = \dim(\mrm{H}^0(A)) - n.
\]
\end{proof}

As a corollary of the above proposition, we obtain the following result,
which we call the double-Cohen–Macaulay theorem.

\begin{thm}\label{thm:DoubleCM}
Let $(A,\bar{\m})$ be a Cohen–Macaulay local DG-ring with constant amplitude,
and suppose that the local ring $\mrm{H}^0(A)$ is also Cohen–Macaulay.
Given a finite sequence of elements $\bar{x}_1,\dots,\bar{x}_n \in \bar{\m}$,
it holds that $\bar{x}_1,\dots,\bar{x}_n$ is an $A$-regular sequence if and only if
$\bar{x}_1,\dots,\bar{x}_n$ is an $\mrm{H}^0(A)$-regular sequence. 
\end{thm}
\begin{proof}
Given a finite sequence of elements $\bar{x}_1,\dots,\bar{x}_n \in \bar{\m}$,
we know by Proposition \ref{prop:sopIsRegular} that $\bar{x}_1,\dots,\bar{x}_n$ is an $A$-regular sequence,
if and only if
\[
\dim(\mrm{H}^0(A)/(\bar{x}_1,\dots, \bar{x}_n)) = \dim(\mrm{H}^0(A)) - n.
\]
By \cite[Theorem 2.1.2(c)]{BH},
this is the case if and only if $\bar{x}_1,\dots,\bar{x}_n$ is an $\mrm{H}^0(A)$-regular sequence.
\end{proof}

Here is the main result of this section.

\begin{thm}\label{thm:main}
Let $(A,\bar{\m})$ be a commutative noetherian local DG-ring with bounded cohomology.
Then the following are equivalent:
\begin{enumerate}
\item The maximal ideal $\bar{\m}$ is generated by an $A$-regular sequence.
\item The DG-ring $A$ is a Cohen–Macaulay DG-ring and $\mrm{H}^0(A)$ is a regular local ring.
\end{enumerate}
\end{thm}
\begin{proof}
Suppose $\bar{\m}$ is generated by an $A$-regular sequence.
Let $n$ be the minimal number of generators of $\bar{\m}$.
Then by definition, $n \le \opn{seq.depth}_A(A)$,
while by \cite[Corollary 5.5]{ShCM}, 
we know that $\opn{seq.depth}_A(A) \le \dim(\mrm{H}^0(A))$.
On the other hand, by Krull's Hauptidealsatz, 
we know that $\dim(\mrm{H}^0(A)) \le n$. 
It follows that $n = \dim(\mrm{H}^0(A))$,
so that $\mrm{H}^0(A)$ is a regular local ring,
and that moreover $\opn{seq.depth}_A(A) = \dim(\mrm{H}^0(A))$,
so that $A$ is local-Cohen–Macaulay.
Since $\mrm{H}^0(A)$ is a regular local ring, 
in particular it is an integral domain,
so by Proposition \ref{prop:const} we deduce that $A$ has constant amplitude.
Being regular, $\mrm{H}^0(A)$ is also catenary.
Hence, by Theorem \ref{thm:lCMisCM} we deduce that $A$ is Cohen–Macaulay.
Conversely, suppose that $A$ is a Cohen–Macaulay DG-ring such that $\mrm{H}^0(A)$ is a regular local ring.
Let $\bar{x}_1,\dots,\bar{x}_n$ be a regular of system of parameters of $\mrm{H}^0(A)$.
In particular, it generates $\bar{\m}$,
and it follows from Theorem \ref{thm:DoubleCM} that it is an $A$-regular sequence. 
\end{proof}

\begin{rem}\label{rem:both-needed}
A reader unacquainted with regular sequences in the DG context might wonder if the Cohen–Macaulay condition is necessary in the above result, or if it simply follows from the condition that $\mrm{H}^0(A)$ is regular.
It is in fact necessary, and without it, this result is far from true.
For instance, if $(A,\m)$ is a regular local ring which is not a field,
and if $\k = A/\m$ is its residue field,
we may consider the trivial extension DG-ring $B = A \skewtimes \k[2]$,
as defined in \cite[Definition 1.2]{JorDual}.
Then by definition $\mrm{H}^0(B) = A$ is a regular local ring,
but since $\mrm{H}^{\inf(B)}(B) = \k$,
it follows that $\opn{seq.depth}_B(B) = 0$.
Thus, $B$ is as far as possible from being Cohen–Macaulay.
Conversely, there are many examples of Cohen–Macaulay DG-rings $A$
such that $\mrm{H}^0(A)$ is not regular.
Thus, there is no redundancy in the above theorem.
\end{rem}

We now make this result into a definition.
\begin{dfn}
Let $(A,\bar{\m})$ be a commutative noetherian local DG-ring with bounded cohomology.
We say that $A$ is sequence-regular if the maximal ideal $\bar{\m}$ is generated by an $A$-regular sequence.
\end{dfn}

In a regular local ring, any minimal generating sequence of the maximal ideal is a regular system of parameters.
Similarly, we have:

\begin{cor}\label{cor:min}
Let $(A,\bar{\m})$ be a commutative noetherian local DG-ring with bounded cohomology.
Then the following are equivalent:
\begin{enumerate}
\item The DG-ring $A$ is sequence-regular, i.e, some set of generators of $\bar{\m}$ is $A$-regular.
\item Any minimal generating sequence of the maximal ideal $\bar{\m}$ is $A$-regular.
\end{enumerate}
\end{cor}
\begin{proof}
It is clear that the second statement implies the first.
For the converse, by Theorem \ref{thm:main},
we know that $A$ is Cohen–Macaulay and that $\mrm{H}^0(A)$ is a regular local ring.
Hence, any minimal generating sequence of the maximal ideal $\bar{\m}$ is $\mrm{H}^0(A)$-regular,
so by the  double-Cohen–Macaulay theorem (Theorem \ref{thm:DoubleCM}),
we deduce that it is also $A$-regular.
\end{proof}

In view of the above, given a sequence-regular DG-ring $(A,\bar{\m})$,
we say that a sequence of elements $\bar{a}_1,\dots,\bar{a}_n \in \bar{\m}$ is a regular system of parameters of $A$
if it is a minimal generating sequence of $\bar{\m}$.
In that case, the above result implies that it is necessarily an $A$-regular sequence.

\section{Behavior of sequence-regularity with respect to standard operations}\label{sec:operations}

In this section we investigate the behavior of the sequence-regularity property with respect to standard operations in derived commutative algebra. We begin by discussing behavior with respect to localization. 
Recall that the Auslander–Buchsbaum-Serre theorem \cite{AusBuch,Serre} states that the localization of a regular local ring is again regular.
We now use Theorem \ref{thm:main} to obtain a DG version of it.

\begin{cor}\label{cor:Serre}
Let $(A,\bar{\m})$ be a commutative noetherian local DG-ring with bounded cohomology which is sequence-regular.
Then for any $\bar{\p} \in \opn{Spec}(\mrm{H}^0(A))$,
the DG-ring $A_{\bar{\p}}$ is sequence-regular.
\end{cor}
\begin{proof}
By Theorem \ref{thm:main}, $A$ is a Cohen–Macaulay DG-ring,
and $\mrm{H}^0(A)$ is a regular local ring.
By the classical theorem of Auslander–Buchsbaum-Serre, we know that $\mrm{H}^0(A_{\bar{\p}}) = \mrm{H}^0(A)_{\bar{\p}}$ is also a regular local ring. 
Since $A$ is a Cohen–Macaulay DG-ring, we know that $A_{\bar{\p}}$ is also Cohen–Macaulay.
Applying Theorem \ref{thm:main} again, we deduce that $A_{\bar{\p}}$ is sequence-regular.
\end{proof}

In view of this result, we may make a global definition:
\begin{dfn}
Let $A$ be a commutative noetherian DG-ring with bounded cohomology.
We say that $A$ is sequence-regular if for any $\bar{\p} \in \opn{Spec}(\mrm{H}^0(A))$,
the local DG-ring $A_{\bar{\p}}$ is sequence-regular.
\end{dfn}

Using this definition, Corollary \ref{cor:Serre} may be stated as follows:
\begin{cor}
Let $A$ be a commutative noetherian DG-ring with bounded cohomology.
Then $A$ is sequence-regular if and only if for all maximal ideals $\bar{\m} \in \opn{Spec}(\mrm{H}^0(A))$,
the local DG-ring $A_{\bar{\m}}$ is sequence-regular.
\end{cor}

As a corollary of the above, we obtain the following global version of Theorem \ref{thm:main}:
\begin{cor}\label{cor:main-global}
Let $A$ be a commutative noetherian DG-ring with bounded cohomology.
Then the following are equivalent:
\begin{enumerate}
\item The DG-ring $A$ is sequence-regular.
\item The DG-ring $A$ is Cohen–Macaulay and the ring $\mrm{H}^0(A)$ is a regular ring.
\end{enumerate}
\end{cor}
\begin{proof}
This follows immediately from Theorem \ref{thm:main},
and the fact that the notions of being a sequence-regular DG-ring, 
being a Cohen–Macaulay DG-ring, and being a regular ring are all local properties.
\end{proof}

Next, we obtain a DG version of the fact that a local ring is regular if and only if its adic completion is.
To do this, we use the derived adic completion functor, introduced in \cite{ShComp}.

\begin{cor}\label{cor:compl}
Let $(A,\bar{\m})$ be a commutative noetherian local DG-ring with bounded cohomology.
Then $A$ is sequence-regular if and only if the derived $\bar{\m}$-adic completion $\mrm{L}\Lambda(A,\bar{\m})$ is sequence-regular.
\end{cor}
\begin{proof}
Let $B = \mrm{L}\Lambda(A,\bar{\m})$.
By \cite[Proposition 4.16]{ShComp},
we know that $\mrm{H}^0(B)$ is simply the $\bar{\m}$-adic completion of $\mrm{H}^0(A)$.
Hence, by \cite[Proposition 2.2.2]{BH}, 
it holds that $\mrm{H}^0(A)$ is a regular local ring if and only if $\mrm{H}^0(B)$ is a regular local ring.
Moreover, by \cite[Proposition 4.6]{ShCM}, the DG-ring $A$ is local-Cohen–Macaulay if and only if $B$ is local-Cohen–Macaulay.
Suppose now that one of $A,B$ is sequence-regular.
By using Theorem \ref{thm:main},
we thus see that $\mrm{H}^0(A)$ and $\mrm{H}^0(B)$ are both catenary,
and have irreducible spectrum.
By Proposition \ref{prop:const} we then deduce that they both have constant amplitude,
and by Theorem \ref{thm:lCMisCM}, we thus see that $A$ and $B$ are both Cohen–Macaulay.
Hence, the result.
\end{proof}

Next we discuss the behavior of sequence-regularity with respect to quotients.
Given a regular local ring $A$,
and given an ideal $I \subseteq A$,
it is well known that $A/I$ is regular if and only if $I$ is generated by a subset of a regular system of parameters of $A$.
See for instance \cite[Proposition 2.2.4]{BH}.
In the DG setting, we replace quotients with Koszul complexes, which are the derived functors of the quotient operation.
Here is a DG version of the above mentioned fact.

\begin{thm}\label{thm:seq-reg}
Let $(A,\bar{\m})$ be a sequence-regular commutative noetherian local DG-ring with bounded cohomology,
and let $\bar{a}_1,\dots,\bar{a}_n \in \bar{\m}$.
Then $K(A;\bar{a}_1,\dots,\bar{a}_n)$ is a sequence-regular DG-ring if and only if
the ideal generated by $\bar{a}_1,\dots,\bar{a}_n$ in $\mrm{H}^0(A)$ 
can be generated by a subset of a regular system of parameters of $A$.
\end{thm}
\begin{proof}
By Theorem \ref{thm:main}, 
we know that $A$ is a Cohen–Macaulay DG-ring with constant amplitude.
Hence, according to \cite[Theorem 4.2]{ShKos},
for any finite sequence of elements $\bar{a}_1,\dots,\bar{a}_n \in \bar{\m}$,
the DG-ring $K(A;\bar{a}_1,\dots,\bar{a}_n)$ is also Cohen–Macaulay.
It follows from using Theorem \ref{thm:main} again that $K(A;\bar{a}_1,\dots,\bar{a}_n)$ 
is sequence-regular if and only if the ring
\[
\mrm{H}^0\left(K(A;\bar{a}_1,\dots,\bar{a}_n)\right) = \mrm{H}^0(A)/(\bar{a}_1,\dots,\bar{a}_n)
\]
is a regular local ring.
As mentioned above, by \cite[Proposition 2.2.4]{BH},
this is the case if and only if the ideal
\[
(\bar{a}_1,\dots,\bar{a}_n) \subseteq \mrm{H}^0(A)
\]
can be generated by a subset of a regular system of parameters of $\mrm{H}^0(A)$,
which, by Corollary \ref{cor:min}, is the same as a subset of a regular system of parameters of $A$.
\end{proof}

\begin{rem}
Let $(A,\m)$ is a noetherian local ring,
and let $a_1,\dots,a_n \in \m$ be any sequence of elements.
If $A$ is Gorenstein, then by \cite[Theorem 4.9]{FJ} (or \cite{AvGol} in case this is a minimal generating sequence of $\m$), 
the derived quotient $K(A;a_1,\dots,a_n)$ is always a Gorenstein DG-ring.
Similarly, if $A$ is Cohen–Macaulay,
by \cite[Theorem 4.2]{ShKos},
the Koszul complex $K(A;a_1,\dots,a_n)$ is always a Cohen–Macaulay DG-ring.
In contrast, if we apply Theorem \ref{thm:seq-reg}
to a regular local ring $(A,\m)$,
we see that $K(A;a_1,\dots,a_n)$ is a sequence-regular DG-ring only in the case where the natural map
\begin{equation}\label{eqn:kosH0}
K(A;a_1,\dots,a_n) \to \mrm{H}^0\left(K(A;a_1,\dots,a_n)\right) = A/(a_1,\dots,a_n)
\end{equation}
is a quasi-isomorphism.
Indeed, by the above theorem, when $K(A;a_1,\dots,a_n)$ is sequence-regular,
the sequence $a_1,\dots,a_n$ is an $A$-regular sequence,
so all the cohomlogies of the Koszul complex except for the top one vanish.
Hence, (\ref{eqn:kosH0}) is a quasi-isomorphism.
The conclusion is that while derived quotients of Gorenstein and Cohen–Macaulay rings are always Gorenstein and Cohen–Macaulay respectively, the derived quotient of a regular local ring is sequence-regular only in cases where it coincide with the ordinary quotient.
\end{rem}

In contrast with the above,
derived quotients of Gorenstein and Cohen–Macaulay rings can often be sequence-regular,
even if they do not coincide with ordinary quotients:
\begin{prop}
Let $A$ be a Cohen–Macaulay ring,
and let $a_1,\dots,a_n \in A$.
Then the Koszul complex $K(A;a_1,\dots,a_n)$ is a sequence-regular DG-ring if and only if
the quotient ring $A/(a_1,\dots,a_n)$ is a regular ring.
\end{prop} 
\begin{proof}
According to \cite[Corollary 4.6]{ShKos},
since $A$ is Cohen–Macaulay, the Koszul complex $K(A;a_1,\dots,a_n)$ is a Cohen–Macaulay DG-ring.
Hence, by Corollary \ref{cor:main-global},
the Koszul complex $K(A;a_1,\dots,a_n)$ is sequence-regular if and only if
\[
\mrm{H}^0\left(K(A;a_1,\dots,a_n)\right) = A/(a_1,\dots,a_n)
\]
is a regular ring.
\end{proof}

This allows one to produce many concrete examples of sequence-regular DG-rings.
For instance, here is a simple example:
\begin{exa}
Let $\K$ be a field, 
and let $A = \K[[x,y]]/(x\cdot y)$.
This ring is Cohen–Macaulay, and even Gorenstein, but is not regular.
Since $x\in A$ is a zero-divisor, 
the Koszul complex $K(A;x)$ has amplitude greater then $0$,
and the ordinary quotient $A/x\cong \K[[y]]$ is a regular local ring,
which implies that $K(A;x)$ is a sequence-regular DG-ring which is not equivalent to a ring.
\end{exa}

Next we wish to study the behavior of the sequence-regular property under flat local homomorphisms.
The correct notion of flatness used here is the notion of a map of flat dimension $0$.
This, by definition, is a map of commutative DG-rings $f:A \to B$,
such that for all $M \in \cat{D}^{\mrm{b}}(A)$,
there is an inequality $\amp(M\otimes^{\mrm{L}}_A B) \le \amp(M)$.
We refer the reader to \cite[Section 2.2.1]{GaRoz} and \cite[Section 2.2.2]{ToVe} for a discussion about this notion.
The important thing to note is that in this case the map $\mrm{H}^0(f):\mrm{H}^0(A) \to \mrm{H}^0(B)$ is also flat,
and that for all $M \in \cat{D}(A)$ and all $n\in \mathbb{N}$,
there is a natural isomorphism $\mrm{H}^n(M\otimes^{\mrm{L}}_A B) \cong \mrm{H}^n(M)\otimes_{\mrm{H}^0(A)} \mrm{H}^0(B)$.
In the local case, we then have the following important property which we will need in the sequel:
\begin{lem}\label{lem:amp-flat}
Let $A,B$ be commutative noetherian local DG-rings,
and let $f:A \to B$ be a map of DG-rings,
such that the induced map $\mrm{H}^0(f):\mrm{H}^0(A)\to \mrm{H}^0(B)$ is local.
Then for any $M \in \cat{D}(A)$,
there is an equality
\[
\amp(M) = \amp(M\otimes^{\mrm{L}}_A B).
\]
\end{lem}
\begin{proof}
Being a flat local map,
the map $\mrm{H}^0(f)$ is faithfully flat.
Since for any $n\in \mathbb{Z}$ we have that
\[
\mrm{H}^n(M\otimes^{\mrm{L}}_A B) \cong \mrm{H}^n(M)\otimes_{\mrm{H}^0(A)} \mrm{H}^0(B),
\]
and since $\mrm{H}^0(B)$ is faithfully flat over $\mrm{H}^0(A)$,
we deduce that $\mrm{H}^n(M\otimes^{\mrm{L}}_A B) = 0$
if and only if $\mrm{H}^n(M) = 0$.
This implies the equality
\[
\amp(M) = \amp(M\otimes^{\mrm{L}}_A B).
\]
\end{proof}

Given a flat local homomorphism $\varphi:(A,\m) \to (B,\n)$ between noetherian local rings,
according to \cite[Theorem 2.2.12]{BH},
if the ring $B$ is regular then $A$ is regular,
while if $A$ and the fiber ring $B/\m B$ are regular, then $B$ is regular.
Here is a DG version of this result.

\begin{thm}\label{thm:transfer-flat}
Let $(A,\bar{\m}), (B,\bar{\n})$ be commutative noetherian local DG-rings with bounded cohomology,
and let $f:A \to B$ be a map of DG-rings of flat dimension $0$,
such that the induced map $\mrm{H}^0(f):(\mrm{H}^0(A),\bar{\m}) \to (\mrm{H}^0(B),\bar{\n})$ is a local homomorphism.
Assume further that $A$ and $B$ have dualizing DG-modules.
\begin{enumerate}
\item If the DG-ring $B$ is sequence-regular, then the DG-ring $A$ is sequence-regular.
\item If the DG-ring $A$ is sequence-regular,
and the fiber $\mrm{H}^0(B)/\bar{\m}\mrm{H}^0(B)$ is regular,
then the DG-ring $B$ is sequence-regular.
\end{enumerate}
\end{thm}
\begin{proof}
\begin{enumerate}[wide, labelwidth=!, labelindent=0pt]
\item Since $f$ is of flat dimension $0$,
the map $\mrm{H}^0(f):\mrm{H}^0(A) \to \mrm{H}^0(B)$ is also flat.
By Theorem \ref{thm:main},
we know that $B$ is Cohen–Macaulay and that $\mrm{H}^0(B)$ is a regular local ring.
Using \cite[Theorem 2.2.12]{BH} applied to the flat map $\mrm{H}^0(f)$,
we deduce that $\mrm{H}^0(A)$ is also a regular local ring.
Being regular, the local ring $\mrm{H}^0(B)$ is also a Gorenstein ring,
so using \cite[Corollary 3.3.15]{BH},
we deduce that the fiber ring $\mrm{H}^0(B)/\bar{\m}\mrm{H}^0(B)$ is Gorenstein.
In other words, the morphism $\mrm{H}^0(f):\mrm{H}^0(A)\to \mrm{H}^0(B)$ is a Gorenstein homomorphism
(in the sense of \cite{AvFox}, but also in the sense of Grothendieck, as this is a classical situation of a flat map).
Since $f$ is of flat dimension $0$, 
by \cite[Theorem 6.9]{ShTwist} we may lift the Gorenstein property and deduce that $f$ itself is also a Gorenstein homomorphism.
By definition this means that if $R$ is a dualizing DG-module over $A$,
then $S = R\otimes^{\mrm{L}}_A B$ is a dualizing DG-module over $B$.
Since $B$ is Cohen–Macaulay,
by \cite[Proposition 4.4]{ShCM},
there is an equality $\amp(S) = \amp(B)$.
Using Lemma \ref{lem:amp-flat} there are equalities $\amp(A) = \amp(B)$,
and $\amp(S) = \amp(R)$.
Hence, $\amp(A) = \amp(R)$,
so using \cite[Proposition 4.4]{ShCM} again,
we deduce that $A$ is local-Cohen–Macaulay.
As $\mrm{H}^0(A)$ is regular, we then deduce that $A$ is Cohen–Macaulay,
so by Theorem \ref{thm:main}, $A$ is sequence-regular.
\item The proof of the second statement is similar:
By Theorem \ref{thm:main},
the ring $\mrm{H}^0(A)$ is regular,
and $A$ is Cohen–Macaulay.
Since $\mrm{H}^0(f)$ is flat and its fiber is regular,
by \cite[Theorem 2.2.12]{BH}, the ring $\mrm{H}^0(B)$ is regular.
Let $R$ be a dualizing DG-module over $A$.
Again, we see that the map $\mrm{H}^0(f)$ is Gorenstein,
so by \cite[Theorem 6.9]{ShTwist}, the map $f$ is again Gorenstein.
Thus, using \cite[Proposition 4.4]{ShCM} and Lemma \ref{lem:amp-flat}, 
we see that
\[
\amp(B) = \amp(A) = \amp(R) = \amp(R\otimes^{\mrm{L}}_A B),
\]
so that $B$ is local-Cohen–Macaulay, 
and as before, regularity of $\mrm{H}^0(B)$ also implies that $B$ is Cohen–Macaulay.
\end{enumerate}
\end{proof}

\begin{rem}\label{rem:dual-complex}
The assumption that $A$ and $B$ have dualizing DG-modules is not necessary in the above result.
There are two ways to remove it.
One way is to pass from $f:A \to B$ to its derived completion (in the sense of \cite{ShComp}).
It is possible to show, using DG versions of results such as those of \cite{YekFlat}
that the derived completion of $f$ is also of flat dimension $0$.
But then the local DG-rings $A,B$ are derived adically complete,
so by \cite[Proposition 7.21]{ShInj},
they have dualizing DG-modules.
Then, the proof above holds, 
and finally one may use Corollary \ref{cor:compl} to deduce sequence-regularity of $A$ or $B$.
Alternatively, we conjecture that any sequence-regular DG-ring has a dualizing DG-module.
This would follow from a DG version of \cite[Theorem 4.3.5]{Lur} which says that for a noetherian $E_{\infty}$-ring $R$,
if $\pi_0(R)$ has a dualizing complex, then $R$ has a dualizing complex.
In our context, if $A$ is sequence-regular, then $\mrm{H}^0(A)$ is a regular local ring,
so in particular $\mrm{H}^0(A)$ is a dualizing complex over itself.
\end{rem}

\section{The residue DG-field of a sequence-regular DG-ring}\label{sec:resField}

Given a commutative noetherian local ring $(A,\m)$,
one may associate to it the residue field $A/\m$,
which is of great importance in commutative algebra. 
In particular, $A$ is regular exactly when $A/\m$ has finite projective dimension over $A$.
Suppose now that $(A,\bar{\m})$ is a commutative noetherian DG-ring with bounded cohomology.
Of course, one may still associate to it the residue field $\mrm{H}^0(A)/\bar{\m}$,
but unlike the case of rings, this is now an object with different amplitude then $A$.
Moreover, considered as an object of $\cat{D}(A)$, the DG-module $\mrm{H}^0(A)/\bar{\m}$
\textbf{never} has finite projective dimension.
In this section we discuss an alternative construction in the DG-setting.

If $(A,\bar{\m})$ is any commutative noetherian DG-ring with bounded cohomology,
and if $\bar{a}_1,\dots,\bar{a}_n$ is a sequence of elements in $\bar{\m}$ that generates $\bar{\m}$,
we may consider the Koszul complex $K = K(A;\bar{a}_1,\dots,\bar{a}_n)$.
By basic properties of the Koszul complex, $K$, considered as an object of $\cat{D}(A)$ has finite projective dimension,
and moreover it is compact.
Furthermore, $\mrm{H}^0(K) = \mrm{H}^0(A)/\bar{\m}$ is a field, 
so one may think of $K$ as a residue DG-field associated to $A$.
However, there are two basic properties that this construction loses over the classical one.
First, usually one has $\amp(K) > \amp(A)$.
In fact, we shall see below that these amplitudes are equal if and only if $A$ is sequence-regular.
A more serious problem is that the above construction of $K$ depends on the choice of the generators of $\bar{\m}$,
and that different choices lead to non-isomorphic objects. 

We shall show below that both of these problems do not occur if $A$ happens to be sequence-regular.
In that case, we will be able to associate to it a canonical residue DG-field, independent of any choices,
and moreover, this association will be functorial with respect to flat regular maps.

We begin with the following simple observation regarding the amplitude of a residue DG-field.
\begin{prop}
Let $(A,\bar{\m})$ be a commutative noetherian DG-ring with bounded cohomology.
Then there exist elements $\bar{a}_1,\dots,\bar{a}_n \in \bar{\m}$
such that $K = K(A;\bar{a}_1,\dots,\bar{a}_n)$ satisfies that $\mrm{H}^0(K) = \mrm{H}^0(A)/\bar{\m}$,
and that $\amp(K) = \amp(A)$ if and only if $A$ is sequence-regular.
In that case, this equality is satisfied for any regular system of parameters of $A$.
\end{prop}
\begin{proof}
Note that $\mrm{H}^0(K) = \mrm{H}^0(A)/\bar{\m}$ if and only if $\bar{a}_1,\dots,\bar{a}_n$ generate $\bar{\m}$.
On the other hand, 
by \cite[Lemma 2.13(2)]{Mi2},
there is an equality $\amp(K) = \amp(A)$ if and only if
the sequence $\bar{a}_1,\dots,\bar{a}_n$ is $A$-regular.
We thus see that these two properties hold if and only if there is an $A$-regular sequence which generates $\bar{\m}$,
which exactly says that $A$ is sequence-regular.
\end{proof}

The following is a basic property regarding the uniqueness of a regular system of parameters in a regular local ring.

\begin{prop}\label{prop:inv-mat}
Let $(A,\m)$ be a regular local ring,
and let $a_1,\dots,a_n$ and $b_1,\dots,b_n$ be two regular systems of parameters of $A$.
Then there exist an invertible matrix $R \in \opn{GL}_n(A)$,
such that 
\[
R\cdot \begin{bmatrix} a_1 \\ \vdots \\ a_n \end{bmatrix} = \begin{bmatrix} b_1 \\ \vdots \\ b_n \end{bmatrix}
\]
\end{prop}
\begin{proof}
Since both $a_1,\dots,a_n$ and $b_1,\dots,b_n$ generate the maximal ideal $\m$,
we may find elements $c_{i,j}, d_{i,j}$ for $1\le i,j \le n$,
such that for all $1\le i \le n$, it holds that
\[
b_i = \sum_{j=1}^n c_{i,j}\cdot a_j, \quad a_i = \sum_{j=1}^n d_{i,j}\cdot b_j.
\]
Thus, letting $R=(c_{i,j}), S=(d_{i,j}) \in M_n(A)$, we see that 
\[
R\cdot \begin{bmatrix} a_1 \\ \vdots \\ a_n \end{bmatrix} = \begin{bmatrix} b_1 \\ \vdots \\ b_n \end{bmatrix}, \quad
S\cdot \begin{bmatrix} b_1 \\ \vdots \\ b_n \end{bmatrix} = \begin{bmatrix} a_1 \\ \vdots \\ a_n \end{bmatrix}.
\]
It remains to show that $R$ is invertible.
Letting $\mathbf{a} = \left[\begin{smallmatrix} a_1 \\ \vdots \\ a_n \end{smallmatrix}\right]$,
and denoting by $I_n$ the identity matrix over $R$ of size $n$,
the above implies that 
\[
(S\cdot R - I_n)\cdot \mathbf{a} = 0.
\]
Denoting by $(e_{i,j})$ the coefficients of the matrix $S\cdot R - I_n$,
we see that for all $1 \le i \le n$,
there is an equality
\[
\sum_{j=1}^n e_{i,j}\cdot a_j = 0.
\]
Since $a_1,\dots,a_n$ is an $A$-regular sequence,
this implies (for instance by the proof of \cite[Theorem 16.1]{Mat}) that $e_{i,j} \in (a_1,\dots,a_n)\cdot A = \m$.
Denoting by $\pi:A\to A/\m$ the natural map, 
we deduce that $\pi(S\cdot R) = \pi(I_n)$,
and hence $\pi(\det(S\cdot R)) = \det(\pi(S\cdot R)) = 1$,
which shows that $\det(S\cdot R)\in 1+\m$. Hence, $S\cdot R$ is invertible, so that $R$ is invertible.
\end{proof}

Given a ring $A$, the Koszul complex of $A$ with respect to a finite sequence of elements $a_1,\dots,a_n$ is invariant with respect to the action of $\opn{GL}_n(A)$. 
Our next result is a DG version of this fact.

\begin{prop}\label{prop:koszulInvMat}
Let $A$ be a commutative DG-ring,
and let $\bar{a}_1,\dots,\bar{a}_n \in \mrm{H}^0(A)$.
Suppose we are given an invertible matrix $\bar{T} \in \opn{GL}_n(\mrm{H}^0(A))$,
and let
\[
\left(\begin{matrix}\bar{b}_1\\ \vdots\\ \bar{b}_n\end{matrix}\right) = \bar{T}\cdot \left(\begin{matrix}\bar{a}_1\\ \vdots \\ \bar{a}_n\end{matrix}\right).
\]
Then there is a natural isomorphism in the homotopy category of DG-algebras over $A$:
\[
K(A;\bar{a}_1,\dots,\bar{a}_n) \cong K(A;\bar{b}_1,\dots,\bar{b}_n).
\]
\end{prop}
\begin{proof}
Denote by $\pi^0_A:A^0 \to \mrm{H}^0(A)$ the degree zero of the natural map $\pi_A:A \to \mrm{H}^0(A)$,
and denote by $\bar{S}\subseteq \mrm{H}^0(A)$ the multiplicatively closed set which consists of all invertible elements in $\mrm{H}^0(A)$. 
Let $S = (\pi^0_A)^{-1}(\bar{S})$. 
Then by definition $S$ is also a multiplicatively closed set.
Consider the localization DG-ring $B = A\otimes_{A^0} S^{-1}A^0$.
By \cite[Proposition 4.6]{YeDual}, 
we know that 
\[
\mrm{H}^n(B) = \mrm{H}^n(A)\otimes_{\mrm{H}^0(A)} \bar{S}^{-1}\mrm{H}^0(A) = \mrm{H}^n(A),
\]
so the natural map $A \to B$ is a quasi-isomorphism.
It follows that we may replace $A$ by $B$, 
and it is enough to prove the theorem for the DG-ring $B$.
The reason for this somewhat funny replacement is that $B$ has the following extra property:
the natural map $\pi^0_B:B^0 \to \mrm{H}^0(B)$ has the property that $b\in B^0$ is invertible if and only if $\pi^0_B(b) \in \mrm{H}^0(B)$ is invertible.
Now, suppose we are given elements $\bar{a}_1,\dots,\bar{a}_n \in \mrm{H}^0(B)$,
and an invertible matrix $T \in \opn{GL}_n(\mrm{H}^0(B))$.
As $\pi^0_B$ is surjective, we may find, for each $1 \le i \le n$,
an element $a_i \in B^0$ such that $\pi^0_B(a_i) = \bar{a}_i$.
Similarly, writing $\bar{T} = (\bar{c}_{i,j})$,
for each $1 \le i,j \le n$,
let us choose an element $c_{i,j} \in B^0$,
such that $\pi^0_B(c_{i,j}) = \bar{c}_{i,j}$,
and gather these elements to a matrix $T = (c_{i,j}) \in M_n(B^0)$.
Since $\pi^0_B$ is a ring homomorphism,
it follows that
\[
\pi^0_B(\det(T)) = \det(\pi^0_B(T)) = \det(\bar{T}),
\]
and by the property of $\pi^0_B$ mentioned above,
the fact that $\det(\bar{T})$ is invertible implies that $\det(T)$ is invertible,
so that $T \in \opn{GL}_n(B^0)$.
Let us set
\[
\left(\begin{matrix}b_1\\ \vdots\\ b_n\end{matrix}\right) = T\cdot \left(\begin{matrix}a_1\\ \vdots \\ a_n\end{matrix}\right).
\]
It follows from our choices of $T$ and $a_1,\dots,a_n$ that for all $1\le i \le n$ there is an equality $\pi^0_B(b_i) = \bar{b}_i$.
Functoriality of the Koszul complex implies that
there is a natural isomorphism
\[
K(B^0;a_1,\dots,a_n) \cong K(B^0;b_1,\dots,b_n).
\]
On the other hand, by \cite[Proposition 2.6]{ShKos}
(or by \cite[Lemma 2.8]{Mi2} if $n=1$),
we know that there are natural isomorphisms
\[
K(B;\bar{a}_1,\dots,\bar{a}_n) \cong K(B^0;a_1,\dots,a_n)\otimes^{\mrm{L}}_{B^0} B, 
\]
and
\[ 
K(B;\bar{b}_1,\dots,\bar{b}_n) \cong K(B^0;b_1,\dots,b_n)\otimes^{\mrm{L}}_{B^0} B
\]
in the homotopy category of DG-algebras over $B$.
Combining all these natural isomorphisms, 
we obtain the required isomorphism
\[
K(B;\bar{a}_1,\dots,\bar{a}_n) \cong K(B;\bar{b}_1,\dots,\bar{b}_n).
\]
\end{proof}

Using the above results, we now construct the residue DG-field associated to a local sequence-regular DG-ring.

\begin{thm}\label{thm:unique-DG-field}
Let $(A,\bar{\m})$ be a sequence-regular commutative noetherian local DG-ring,
and let $\bar{a}_1,\dots,\bar{a}_n$ by a regular system of parameters of $A$.
\begin{enumerate}
\item If $\hat{a}_1,\dots,\hat{a}_n$ is another regular system of parameters of $A$,
then there is a natural isomorphism in the homotopy category of DG-algebras over $A$:
\[
K(A;\bar{a}_1,\dots,\bar{a}_n) \cong K(A;\hat{a}_1,\dots,\hat{a}_n)
\]
\item If $(B,\bar{\n})$ is another sequence-regular noetherian local DG-ring,
$f:A \to B$ is a local homomorphism of DG-rings,
such that $\mrm{H}^0(f)$ is flat and the fiber $\mrm{H}^0(b)/\bar{\n}$ is a regular ring,
then there exist a regular system of parameters $\bar{b}_1,\dots,\bar{b}_m$ of $B$
and a map of DG-rings $K(A;\bar{a}_1,\dots,\bar{a}_n) \to K(B;\bar{b}_1,\dots,\bar{b}_m)$
making the diagram
\[
\begin{tikzcd}
A \arrow[r, "f"] \arrow[d] & B \arrow[d]\\
K(A;\bar{a}_1,\dots,\bar{a}_n) \arrow[r] & K(B;\bar{b}_1,\dots,\bar{b}_m)
\end{tikzcd}
\]
commutative, 
where the vertical maps are the natural maps from a DG-ring to its Koszul complex.
\end{enumerate}
\end{thm}
\begin{proof}
\begin{enumerate}[wide, labelwidth=!, labelindent=0pt]
\item Since $A$ is sequence-regular,
by Theorem \ref{thm:main} and the double-Cohen-Macaulay theorem,
any regular system of parameters of $A$ is also a regular system of parameters of the regular local ring $\mrm{H}^0(A)$.
If $\bar{a}_1,\dots,\bar{a}_n$ and $\hat{a}_1,\dots,\hat{a}_n$ are two such regular systems of parameters,
by Proposition \ref{prop:inv-mat} we may find an invertible matrix $\bar{T} \in \opn{GL}_n(\mrm{H}^0(A))$
which maps $\bar{a}_1,\dots,\bar{a}_n$ to $\hat{a}_1,\dots,\hat{a}_n$.
Then, applying Proposition \ref{prop:koszulInvMat},
we obtain the required natural isomorphism
\[
K(A;\bar{a}_1,\dots,\bar{a}_n) \cong K(A;\hat{a}_1,\dots,\hat{a}_n).
\]
\item Since $\mrm{H}^0(f):\mrm{H}^0(A) \to \mrm{H}^0(B)$ is a flat map with a regular fiber,
by the proof of \cite[Theorem 2.2.12(b)]{BH},
the sequence $\mrm{H}^0(f)(\bar{a}_1),\dots,\mrm{H}^0(f)(\bar{a}_n)$ is part of a regular system of parameters of the regular local ring $\mrm{H}^0(B)$.
Denoting this sequence by $\bar{b}_1,\dots,\bar{b}_n$,
it follows that we may complete it to a regular system of parameters $\bar{b}_1,\dots,\bar{b}_n,\dots,\bar{b}_m$ of $\mrm{H}^0(B)$,
and hence, also of $B$.
By functoriality of Koszul complexes, this implies that there is a natural map $K(A;\bar{a}_1,\dots,\bar{a}_n) \to K(B;\bar{b}_1,\dots,\bar{b}_n)$ making the diagram
\[
\begin{tikzcd}
A \arrow[r, "f"] \arrow[d] & B \arrow[d]\\
K(A;\bar{a}_1,\dots,\bar{a}_n) \arrow[r] & K(B;\bar{b}_1,\dots,\bar{b}_n)
\end{tikzcd}
\]
commutative. 
Finally, the fact that the diagram
\[
\begin{tikzcd}
B \arrow[d] \arrow[rd] &  \\
K(B;\bar{b}_1,\dots,\bar{b}_n) \arrow[r] & K(B;\bar{b}_1,\dots,\bar{b}_m)
\end{tikzcd}
\]
is commutative implies that we may obtain the required diagram by composing the map
$K(A;\bar{a}_1,\dots,\bar{a}_n) \to K(B;\bar{b}_1,\dots,\bar{b}_n)$ with the map
\[
K(B;\bar{b}_1,\dots,\bar{b}_n) \to K(B;\bar{b}_1,\dots,\bar{b}_m).
\]
\end{enumerate}
\end{proof}

In view of Theorem \ref{thm:unique-DG-field}, we make the following definition:
\begin{dfn}
Let $(A,\bar{\m})$ be a sequence-regular commutative noetherian local DG-ring.
We define the residue DG-field of $A$ to be the DG-algebra over $A$ given by 
\[
K(A;\bar{a}_1,\dots,\bar{a}_n),
\]
where $\bar{a}_1,\dots,\bar{a}_n$ is some regular system of parameters of $A$.
We denote the residue DG-field of $A$ by $\kappa(A)$.
\end{dfn}

As mentioned above, since $\kappa(A)$ is given as a Koszul complex, 
it must be compact. 
We can be more explicit, and compute its flat dimension.
Recall that the flat dimension of the DG-module $M$ over a DG-ring $A$ is given by
\[
\opn{flat\dim}_A(M) = \inf\{n\in\mathbb{Z}\mid\opn{Tor}^j_A(M,N) = 0\text{, }\forall N \in \cat{D}^{\mrm{b}}(A)\text{, } \forall j>n-\inf N\}
\] 
If $(A,\m)$ is a regular local ring, it is well known that the residue field $A/\m$ has flat dimension equal to the Krull dimension $\dim(A)$.
Similarly, we have:
\begin{prop}\label{prop:flatDim}
Let $(A,\bar{\m})$ be a sequence-regular commutative noetherian local DG-ring.
Suppose that the Krull dimension of the ring $\mrm{H}^0(A)$ is equal to $d$.
Then 
\[
\opn{flat\dim}_A(\kappa(A)) = d.
\]
\end{prop}
\begin{proof}
According to \cite[Theorem 4.1]{ShHomDim},
there is an equality
\[
\opn{flat\dim}_A(\kappa(A)) = \opn{flat\dim}_{\mrm{H}^0(A)}(\kappa(A)\otimes^{\mrm{L}}_A \mrm{H}^0(A)).
\]
To compute the latter, let $\bar{a}_1,\dots,\bar{a}_n$ be a regular system of parameters of $A$.
Then by definition we have that
\[
\kappa(A)\otimes^{\mrm{L}}_A \mrm{H}^0(A) = K(A;\bar{a}_1,\dots,\bar{a}_n)\otimes^{\mrm{L}}_A \mrm{H}^0(A).
\]
By the base change property of the Koszul complex, \cite[Proposition 2.9]{ShKos},
applied to the map $\pi_A:A\to \mrm{H}^0(A)$, we have that
\[
K(A;\bar{a}_1,\dots,\bar{a}_n)\otimes^{\mrm{L}}_A \mrm{H}^0(A) \cong K(\mrm{H}^0(A);\bar{a}_1,\dots,\bar{a}_n).
\]
Since $\bar{a}_1,\dots,\bar{a}_n$ is also a regular system of parameters of $\mrm{H}^0(A)$,
in particular it is an $\mrm{H}^0(A)$-regular sequence, so there is an isomorphism
\[
K(\mrm{H}^0(A);\bar{a}_1,\dots,\bar{a}_n) \cong \mrm{H}^0(A)/(\bar{a}_1,\dots,\bar{a}_n) = \mrm{H}^0(A)/\bar{\m}.
\]
It follows that 
\[
\opn{flat\dim}_A(\kappa(A)) = \opn{flat\dim}_{\mrm{H}^0(A)}(\kappa(\mrm{H}^0(A))) = d.
\]
\end{proof}

In the course of proving the above result, 
we have also established the following useful property, 
so we record it separately:
\begin{prop}\label{prop:redKA}
Let $(A,\bar{\m})$ be a sequence-regular commutative noetherian local DG-ring.
Then $\kappa(A)\otimes^{\mrm{L}}_A \mrm{H}^0(A) \cong \kappa(\mrm{H}^0(A))$.
\end{prop}

We can use this property of $\kappa(A)$ to give another characterization of sequence-regular DG-rings:
\begin{thm}\label{thm:charac-residue}
Let $(A,\bar{\m})$ be a commutative noetherian local DG-ring with bounded cohomology and 
a noetherian model.
Then $A$ is sequence-regular if and only if there exist a commutative noetherian DG-ring $K$,
and a map of DG-rings $A \to K$ such that $K$ is compact over $A$,
there is an equality $\amp(K) = \amp(A)$,
and such that there is a DG-ring isomorphism $K\otimes^{\mrm{L}}_A \mrm{H}^0(A) \cong \mrm{H}^0(A)/\m$.
\end{thm}
\begin{proof}
If $A$ is sequence-regular, then we already saw that $K=\kappa(A)$ satisfies all these properties.
Conversely, suppose that there is such a $K$.
Since $K$ is compact over $A$, by \cite[Theorems 5.11 and 5.20]{YeDual},
we have that 
\[
K\otimes^{\mrm{L}}_A \mrm{H}^0(A) \cong \mrm{H}^0(A)/\m
\]
is compact over $\mrm{H}^0(A)$.
This implies by the Auslander–Buchsbaum-Serre theorem that $\mrm{H}^0(A)$ is a regular local ring.
We further note that by the proof of \cite[Proposition 3.1]{YeDual},
we have that 
\[
\mrm{H}^0(K) = \mrm{H}^0\left(K\otimes^{\mrm{L}}_A \mrm{H}^0(A)\right) = \mrm{H}^0(A)/\m.
\]
Since $K$ is noetherian, and since for all $n\in \mathbb{Z}$, 
the $\mrm{H}^0(A)$-action on $\mrm{H}^n(K)$ factors through the action of the field $\mrm{H}^0(K)$,
we deduce that $\mrm{H}^n(K)$ has finite length over $\mrm{H}^0(A)$.
Let $E=E(A,\bar{\m})$ be the indecomposable injective DG-module which corresponds to the maximal ideal $\bar{\m}$,
in the sense of \cite[Section 7.2]{ShInj}.
By definition, this is a DG-module $E \in \cat{D}(A)$ such that $\mrm{H}^0(E)$ is equal to the injective hull 
$E(\mrm{H}^0(A),\bar{\m})$ 
of $\mrm{H}^0(A)/\bar{\m}$ over $\mrm{H}^0(A)$, and such that $E$ has injective dimension $0$ over $A$.
Let $N = \mrm{R}\opn{Hom}_A(K,E)$.
By \cite[Theorem 4.10]{ShInj}, for all $n\in \mathbb{Z}$ it holds that
\[
\mrm{H}^n(N) = \mrm{H}^n\left(\mrm{R}\opn{Hom}_A(K,E)\right) \cong \opn{Hom}_{\mrm{H}^0(A)}(\mrm{H}^{-n}(K),\mrm{H}^0(E)).
\]
Since $\mrm{H}^0(E)$ is a faithfully injective module over $\mrm{H}^0(A)$,
this implies that $\amp(N) = \amp(K) = \amp(A)$.
Moreover, since $\mrm{H}^{-n}(K)$ has finite length over $\mrm{H}^0(A)$,
Matlis duality implies that $\mrm{H}^n(N)$ also has finite length,
and in particular it is finitely generated over $\mrm{H}^0(A)$.
Let $M \in \cat{D}(A)$.
The fact that $K$ is compact implies by \cite[Theorem 14.1.22]{YeBook} that
\[
\mrm{R}\opn{Hom}_A(M,\mrm{R}\opn{Hom}_A(K,E)) \cong \mrm{R}\opn{Hom}_A(M,\mrm{R}\opn{Hom}_A(K,A)\otimes^{\mrm{L}}_A E).
\]
Our goal in writing the above isomorphism is to show that $N$ has finite injective dimension.
To continue this chain of isomorphisms, we claim that $\mrm{R}\opn{Hom}_A(K,A)$ is compact.
This follows from Proposition \ref{prop:hom-compact} below.
Since $\mrm{R}\opn{Hom}_A(K,A)$ is compact, by \cite[Proposition 2.18]{ShKos},
there is an isomorphism
\[
\mrm{R}\opn{Hom}_A(M,\mrm{R}\opn{Hom}_A(K,A)\otimes^{\mrm{L}}_A E) \cong
\mrm{R}\opn{Hom}_A(M,E) \otimes^{\mrm{L}}_A \mrm{R}\opn{Hom}_A(K,A).
\]
Combining the above, we see that for any $M \in \cat{D}(A)$,
there is an isomorphism
\[
\mrm{R}\opn{Hom}_A(M,\mrm{R}\opn{Hom}_A(K,E)) \cong \mrm{R}\opn{Hom}_A(M,E) \otimes^{\mrm{L}}_A \mrm{R}\opn{Hom}_A(K,A),
\]
so the fact that $E$ has finite injective dimension and $\mrm{R}\opn{Hom}_A(K,A)$ has finite flat dimension over $A$,
implies that $N = \mrm{R}\opn{Hom}_A(K,E)$ has finite injective dimension over $A$.
This, together with the facts, already established, that $\amp(N) = \amp(A)$ and $N \in \cat{D}^{\mrm{b}}_{\mrm{f}}(A)$,
imply by \cite[Theorem 5.22(2)]{ShCM} that $A$ is local-Cohen–Macaulay.
Since $\mrm{H}^0(A)$ is a catenary integral domain,
as we saw before, this implies that $A$ is Cohen–Macaulay,
so by Theorem \ref{thm:main}, we conclude that $A$ is sequence-regular.
\end{proof}

In the course of the above proof, we needed the following basic fact about compact objects over commutative DG-rings:
\begin{prop}\label{prop:hom-compact}
Let $A$ be a commutative DG-ring, 
and let $M,N \in \cat{D}(A)$ be compact DG-modules.
Then $\mrm{R}\opn{Hom}_A(M,N)$ is also compact.
\end{prop}
\begin{proof}
Since $M$ is compact over $A$,
by \cite[Theorem 14.1.22]{YeBook} there is an isomorphism
\[
\mrm{R}\opn{Hom}_A(M,N) \otimes^{\mrm{L}}_A \mrm{H}^0(A) \cong \mrm{R}\opn{Hom}_A(M,N \otimes^{\mrm{L}}_A \mrm{H}^0(A)).
\]
By adjunction, there is also an isomorphism
\[
\mrm{R}\opn{Hom}_A(M,N \otimes^{\mrm{L}}_A \mrm{H}^0(A)) \cong
\mrm{R}\opn{Hom}_{\mrm{H}^0(A)}(M \otimes^{\mrm{L}}_A \mrm{H}^0(A),N \otimes^{\mrm{L}}_A \mrm{H}^0(A)).
\]
According to \cite[Theorem 5.11]{YeDual},
the complexes $M \otimes^{\mrm{L}}_A \mrm{H}^0(A)$ and $N \otimes^{\mrm{L}}_A \mrm{H}^0(A)$ are both compact over $\mrm{H}^0(A)$,
so by the corresponding result over rings, which easily follows from the characterization of being compact has having bounded projective resolution composed of finitely generated projectives,
we deduce that 
\[
\mrm{R}\opn{Hom}_{\mrm{H}^0(A)}(M \otimes^{\mrm{L}}_A \mrm{H}^0(A),N \otimes^{\mrm{L}}_A \mrm{H}^0(A))
\]
is compact over $\mrm{H}^0(A)$.
Then, the fact that it is isomorphic to $\mrm{R}\opn{Hom}_A(M,N) \otimes^{\mrm{L}}_A \mrm{H}^0(A)$
implies, again by \cite[Theorem 5.11]{YeDual},
that $\mrm{R}\opn{Hom}_A(M,N)$ is compact over $A$.
\end{proof}

The residue DG-field satisfies Nakayama's lemma:
\begin{prop}\label{prop:Nak}
Let $(A,\bar{\m})$ be a sequence-regular noetherian local DG-ring,
and let $M \in \cat{D}^{-}_{\mrm{f}}(A)$.
Then $M \cong 0$ if and only if $\kappa(A)\otimes^{\mrm{L}}_A M \cong 0$.
\end{prop}
\begin{proof}
Suppose $M\ncong 0$, and let $n = \sup(M)$.
Then $\mrm{H}^n(M)$ is a non-zero finitely generated module over $\mrm{H}^0(A)$.
Then by \cite[Proposition 3.1]{YeDual} and its proof,
we have that
\[
\mrm{H}^n(\kappa(A)\otimes^{\mrm{L}}_A M) \cong \mrm{H}^0(\kappa(A))\otimes_{\mrm{H}^0(A)} \mrm{H}^n(M) \cong 
\mrm{H}^0(A)/\bar{\m}\otimes_{\mrm{H}^0(A)} \mrm{H}^n(M)
\]
and the latter is non-zero by the classical Nakayama lemma. 
\end{proof}

Given a noetherian local ring $(A,\m)$,
an important property of the residue field $A/\m$ is that it can detect compact objects.
A finitely generated $A$-module $M$ is compact if and only if $(A/\m)\otimes^{\mrm{L}}_A M$ has finite amplitude.
Such a characterization cannot hold in our context, 
because any non-trivial DG-ring has infinite global dimension, 
and by our construction, the residue DG-field $\kappa(A)$ always has finite flat dimension,
so that $\kappa(A)\otimes^{\mrm{L}}_A M$ has finite amplitude whenever $M$ has.
There is however an alternative formulation of the above fact that does generalize to our context:
instead of checking whether $(A/\m)\otimes^{\mrm{L}}_A M$ has finite amplitude,
one can instead consider this as object of $\cat{D}(A/\m)$,
and ask if it is a compact object there.
Then, the usual characterization still holds: 
a bounded above complex with finitely generated cohomology over a field has finite amplitude if and only if it is compact.
Thus, over a local ring $(A,\m)$, a finitely generated module $M$ is compact if and only if $(A/\m)\otimes^{\mrm{L}}_A M$ is compact over $A/\m$.
This generalizes to our context as follows:

\begin{prop}\label{prop:Com}
Let $(A,\bar{\m})$ be a sequence-regular noetherian local DG-ring,
and let $M \in \cat{D}^{\mrm{b}}_{\mrm{f}}(A)$.
Then $M$ is compact if and only if $\kappa(A)\otimes^{\mrm{L}}_A M \in \cat{D}(\kappa(A))$ is compact.
\end{prop}
\begin{proof}
If $M$ is compact then $\kappa(A)\otimes^{\mrm{L}}_A M$ is compact in $\cat{D}(\kappa(A))$ by \cite[Proposition 5.3]{YeDual}.
Conversely, suppose $\kappa(A)\otimes^{\mrm{L}}_A M$ is compact in $\cat{D}(\kappa(A))$.
Then by \cite[Proposition 5.3]{YeDual} we know that
\[
\mrm{H}^0(\kappa(A))\otimes^{\mrm{L}}_{\kappa(A)} \kappa(A)\otimes^{\mrm{L}}_A M \cong \mrm{H}^0(A)/\bar{\m} \otimes^{\mrm{L}}_A M
\]
is compact in $\cat{D}(\mrm{H}^0(A)/\bar{\m})$.
Considering the object $N = M\otimes^{\mrm{L}}_A \mrm{H}^0(A) \in \cat{D}^{-}_{\mrm{f}}(\mrm{H}^0(A))$,
we thus see that 
\[
N\otimes^{\mrm{L}}_{\mrm{H}^0(A)} \mrm{H}^0(A)/\m \cong \mrm{H}^0(A)/\bar{\m} \otimes^{\mrm{L}}_A M
\]
is compact in $\cat{D}(\mrm{H}^0(A)/\m)$ is compact,
and as explained in the discussion preceding this result, 
this implies that $N$ is compact in $\cat{D}(\mrm{H}^0(A))$.
By \cite[Theorem 5.11]{YeDual},
this implies that $M$ is compact in $\cat{D}(A)$.
\end{proof} 

\begin{rem}
In Propositions \ref{prop:Nak} and \ref{prop:Com} we did not actually need the fact that $A$ is sequence-regular.
The only use of the sequence-regualr property in the above proofs was
in the usage of the canonical $\kappa(A)$,
which does not exist otherwise.
In both of these propositions,
we could alternatively take any set of generators $\bar{a}_1,\dots,\bar{a}_n$ of $\bar{\m}$,
and use $K(A;\bar{a}_1,\dots,\bar{a}_n)$ instead of $\kappa(A)$.
The proofs work the same in this more general context for any noetherian local DG-ring $(A,\bar{\m})$.
\end{rem}

If $A$ is a commutative noetherian ring,
then more generally then above, 
for each prime ideal $\p \in \opn{Spec}(A)$,
there is the associated residue field $\kappa(\p) = A_{\p}/\p A_{\p}$.
Since the localization of a sequence-regular DG-ring is again sequence-regular,
we may thus make the following definition:
\begin{dfn}
Let $A$ be a commutative noetherian sequence-regular DG-ring,
and let $\bar{\p} \in \opn{Spec}(\mrm{H}^0(A))$.
We define the residue DG-field of $A$ at $\bar{\p}$ to be the DG-algebra over $A$,
denoted by $\kappa(A,\bar{\p})$, which is given by
$\varphi_*(\kappa(A_{\bar{\p}}))$,
where $\varphi_*:\cat{D}(A_{\bar{\p}})\to \cat{D}(A)$ is the forgetful functor along the localization map $\varphi:A\to A_{\bar{\p}}$.
\end{dfn}

Generalizing Proposition \ref{prop:redKA}, we have the following basic fact about the residue DG-fields at $\bar{\p}$:
\begin{prop}\label{prop:redP}
Let $A$ be a commutative noetherian sequence-regular DG-ring,
and let $\bar{\p} \in \opn{Spec}(\mrm{H}^0(A))$.
Then $\kappa(A,\bar{\p})\otimes^{\mrm{L}}_A \mrm{H}^0(A) \cong \kappa(\mrm{H}^0(A),\bar{\p}) = \mrm{H}^0(A)_{\bar{\p}}/\bar{\p}\cdot\mrm{H}^0(A)_{\bar{\p}}$.
\end{prop}
\begin{proof}
There are isomorphisms
\[
\kappa(A,\bar{\p})\otimes^{\mrm{L}}_A \mrm{H}^0(A) \cong \kappa(A,\bar{\p})\otimes^{\mrm{L}}_{A_{\bar{\p}}} A_{\bar{\p}} \otimes^{\mrm{L}}_A \mrm{H}^0(A) \cong \kappa(A,\bar{\p})\otimes^{\mrm{L}}_{A_{\bar{\p}}} \mrm{H}^0(A_{\bar{\p}}).
\]
Since $A_{\bar{\p}}$ is a sequence-regular noetherian local DG-ring,
the result follows from Proposition \ref{prop:redKA}.
\end{proof}

Since $\kappa(A,\bar{\p})$ usually does not have finitely generated cohomology,
it is of course not always compact, 
but nevertheless, it does have finite flat dimension,
a fact that follows from Proposition \ref{prop:redP}.
More explicitly, its flat dimension is given in the next result.
The proof of it is almost the same as the proof of Proposition \ref{prop:flatDim}, so we omit it.
\begin{prop}
Let $A$ be a commutative noetherian sequence-regular DG-ring,
and let $\bar{\p} \in \opn{Spec}(\mrm{H}^0(A))$.
Then 
\[
\opn{flat\dim}_A\left(\kappa(A,\bar{\p})\right) = \dim(\mrm{H}^0(A)_{\bar{\p}}) = \opn{ht}(\bar{\p}).
\]
\end{prop}

Given a sequence-regular DG-ring $A$,
and a DG-module $M \in \cat{D}(A)$,
following \cite[Section 2]{FoxBound} and \cite[Section 3]{SSW},
one may now define the small support of $M$ to be the set
\begin{equation}\label{eqn:supp}
\opn{supp}_A(M) = \{\bar{\p}\in \opn{Spec}(\mrm{H}^0(A))\mid M\otimes^{\mrm{L}}_A \kappa(A,\bar{\p}) \ncong 0\}.
\end{equation}
We further remark that it follows from \cite[Theorem 4.5]{ShWi} that this coincides with the set of primes $\bar{\p}$,
such that $M\otimes^{\mrm{L}}_A \mrm{H}^0(A_{\bar{\p}})/\bar{\p}\cdot \mrm{H}^0(A_{\bar{\p}}) \ncong 0$.
However, (\ref{eqn:supp}) has the advantage that $\kappa(A,\bar{\p})$ has finite flat dimension
(while $\mrm{H}^0(A_{\bar{\p}})/\bar{\p}\cdot \mrm{H}^0(A_{\bar{\p}})$ \textbf{always} has infinite flat dimension over $A$),
making computations of these derived tensor products much easier.

We now return to the theme of Theorem \ref{thm:transfer-flat}.
Recall that if $\varphi:(A,\m)\to (B,\n)$ is a local homomorphism between noetherian local rings,
then the fiber ring of $\varphi$ is defined to be $\opn{Fib}(\varphi) = A/\m\otimes_A B$.
Following \cite{AvFox}, the homotopy fiber of $\varphi$ is defined to be the DG-ring $\opn{HFib}(\varphi) = A/\m\otimes^{\mrm{L}}_A B$.
Suppose now more generally that $(A,\bar{\m})$ and $(B,\bar{\n})$ are commutative noetherian local DG-rings,
and suppose $\varphi:A\to B$ is a map of DG-rings. 
Recall that $\varphi$ is called local if $\mrm{H}^0(\varphi)$ is local.
Now, assume further that $A$ is sequence-regular.
Then by the results of this section, we may define the homotopy fiber of $\varphi$ to be
\[
\opn{HFib}(\varphi) := \kappa(A)\otimes^{\mrm{L}}_A B.
\]
The next result is a generalization of \cite[Corollary 5.1]{ShKos},
where we proved the same result under the assumption that $A$ is a regular local ring:

\begin{prop}
Let $\varphi:(A,\bar{\m})\to (B,\bar{\n})$ be a local homomorphism between noetherian local DG-rings,
and assume that $A$ is sequence-regular,
and that $B$ Cohen–Macaulay with constant amplitude.
Then the homotopy fiber $\opn{HFib}(\varphi) = \kappa(A)\otimes^{\mrm{L}}_A B$ is a Cohen–Macaulay DG-ring.
\end{prop} 
\begin{proof}
Let $\bar{a}_1,\dots,\bar{a}_n$ be a regular system of parameters of $A$,
and for each $1\le i \le n$, let $\bar{b}_i = \mrm{H}^0(f)(\bar{a}_i)$.
Then by the definition of $\kappa(A)$ and the base change property of the Koszul complex, 
we have that
\[
\kappa(A)\otimes^{\mrm{L}}_A B = K(A;\bar{a}_1,\dots,\bar{a}_n) \otimes^{\mrm{L}}_A B \cong K(B;\bar{b}_1,\dots,\bar{b}_n).
\]
By \cite[Theorem 4.2]{ShKos}, 
the latter is a Cohen–Macaulay DG-ring,
which proves the result.
\end{proof}

We finish this section with a short discussion about the structure of the derived category of residue DG-fields.
Let $\K$ be a commutative noetherian DG-ring with bounded cohomology such that $\mrm{H}^0(K)$ is a field.
These are precisely the DG-rings that arise as residue DG-fields of sequence-regular local DG-rings.
Note that by definition, any such $\K$ is itself sequence-regular, and its residue DG-field coincide with it.
When $\K = \mrm{H}^0(\K)$ is a field, the derived category $\cat{D}(\K)$ is of course very simple.
It is an abelian category, and every object of $\cat{D}^{\mrm{b}}_{\mrm{f}}(\K)$ is compact.
Furthermore, the only non-zero localizing (respectively colocalizing) subcategory of $\cat{D}(\K)$ is $\cat{D}(\K)$ itself.
Here, a localizing (resp. colocalizing) subcategory is a thick full triangulated subcategory which is closed under arbitrary coproducts (resp. products).

Now suppose again that $\K$ be a commutative noetherian DG-ring with bounded cohomology such that $\mrm{H}^0(K)$ is a field,
but now assume that $\K \ne \mrm{H}^0(\K)$, so that $\K$ is a proper DG-ring not equivalent to a field.
In that case (under the mild assumption that $\K$ has a noetherian model), 
we know by \cite[Theorem 0.2]{Jo} that the set of compact objects of $\cat{D}(\K)$ is \textbf{strictly contained} 
in $\cat{D}^{\mrm{b}}_{\mrm{f}}(\K)$. 
Indeed, any $M \in \cat{D}^{\mrm{b}}_{\mrm{f}}(\K)$ such that $\amp(M) < \amp(K)$ is not compact,
and since by assumption $\amp(K)>0$, there are always such objects $M$.
This failure reflects the fact that the finite structure of $\cat{D}(\K)$ is different then the finite structure of $\cat{D}(\mrm{H}^0(\K))$. 
By that we mean that compact objects of $\cat{D}(\K)$ (resp. $\cat{D}(\mrm{H}^0(\K))$) are precisely those DG-modules that can be finitely built from $\K$ (resp. $\mrm{H}^0(\K)$), and these are in general different.
On the other hand, the infinite structure of these categories is the same, in the following sense.

\begin{prop}
Let $\K$ be a commutative DG-ring with bounded cohomology such that $\mrm{H}^0(\K)$ is a field.
Then the only non-zero localizing (resp. colocalizing) subcategory of $\cat{D}(\K)$ is equal to $\cat{D}(\K)$.
\end{prop}
\begin{proof}
By \cite[Theorems 4.15, 4.17]{ShWi}, there is a bijection between localizing (resp. colocalizing) subcategories of $\cat{D}(\K)$ and subsets of $\opn{Spec}(\mrm{H}^0(K))$ given by sending every such category to the union of its supports (resp. cosupports). 
Since $\opn{Spec}(\mrm{H}^0(K))$ is a singleton, the result follows.
\end{proof}

\begin{rem}
In the paper \cite{Mathew},
Mathew develops a theory of residue fields over certain rational noetherian $E_{\infty}$-rings.
However, the setting in which he works is quite different then ours,
and his construction of residue fields is also completely different.
The main advantage of the residue DG-fields introduced here is,
because of the regularity assumption, they are compact as objects over $A$,
as in the classical situation of regular local rings.
\end{rem}

\section{Generic sequence-regularity}\label{sec:generic}

In this section we show that any eventually coconnective derived algebraic variety over a perfect field is generically sequence-regular.
Given a commutative noetherian ring $A$, 
recall that the regular locus of $A$ is defined as 
\[
\opn{Reg}(A) = \{\p \in \opn{Spec}(A)\mid A_{\p} \mbox{ is a regular local ring}\}.
\]
Recall that a commutative noetherian ring $A$ is called a J-0 ring if $\opn{Reg}(A)$ contains a non-empty open set,
and is called a J-1 ring if $\opn{Reg}(A)$ is an open set.
If $A$ is an integral domain which is J-1, then it is also J-0.

For a commutative noetherian DG-ring $A$, we similarly define its sequence-regular locus by
\[
\opn{seq-Reg}(A) = \{\bar{\p} \in \opn{Spec}(A)\mid A_{\bar{\p}} \mbox{ is a sequence-regular DG-ring}\}.
\]
When a DG-ring $A$ is an ordinary commutative noetherian ring,
we know that there is an equality $\opn{seq-Reg}(A) = \opn{Reg}(A)$.
The next result shows that under mild hypothesis, the set $\opn{seq-Reg}(A)$ is a dense open subset of $\opn{Spec}(\mrm{H}^0(A))$.

\begin{thm}\label{thm:generic}
Let $A$ be a commutative noetherian DG-ring with bounded cohomology,
and suppose that $A$ has a dualizing DG-module,
and that moreover the ring $\mrm{H}^0(A)$ is an a J-1 ring which is an integral domain.
Then the set $\opn{seq-Reg}(A)$ is a dense open subset of $\opn{Spec}(\mrm{H}^0(A))$.
\end{thm}
\begin{proof}
By assumption, the set $\opn{Reg}(\mrm{H}^0(A))$ is a dense open subset of $\opn{Spec}(\mrm{H}^0(A))$.
On the other hand, by the proof of \cite[Theorem 12]{ShOpLoc},
the fact that $\mrm{H}^0(A)$ has an irreducible spectrum,
and that $A$ has a dualizing DG-module implies that the set
\[
\opn{CM}(A) = \{\bar{\p} \in \opn{Spec}(\mrm{H}^0(A))\mid A_{\bar{\p}} \mbox{ is a Cohen–Macaulay DG-ring}\}
\]
is also a dense open subset of $\opn{Spec}(\mrm{H}^0(A))$.
According to Theorem \ref{thm:main}, 
there is an equality
\[
\opn{seq-Reg}(A) = \opn{Reg}(\mrm{H}^0(A)) \cap \opn{CM}(A),
\]
and as both of these sets are dense open subsets, we deduce that $\opn{seq-Reg}(A)$ is also a dense open subset.
\end{proof}

In classical algebraic geometry, any algebraic variety over a perfect field is generically regular.
Similarly, we obtain the following version of this fact in derived algebraic geometry:

\begin{cor}\label{cor:generic}
Let $\K$ be a perfect field,
let $A$ be a commutative noetherian DG-algebra over $\K$ with bounded cohomology,
and suppose that the induced map $\K \to \mrm{H}^0(A)$ is of finite type,
and that $\mrm{H}^0(A)$ is an integral domain.
Then $\opn{seq-Reg}(A)$ is a dense open subset of $\opn{Spec}(\mrm{H}^0(A))$.
\end{cor}
\begin{proof}
Under these assumptions, $\opn{Spec}(\mrm{H}^0(A))$ is simply a classical algebraic variety over $\K$.
Moreover, the assumption that $\K \to \mrm{H}^0(A)$ is of finite type,
implies by \cite[Theorem 7.9]{YeDual} that $A$ has a dualizing DG-module.
Hence, the assumptions of Theorem \ref{thm:generic} are satisfied.
\end{proof}

In geometric terms, one may formulate the above as:
\begin{cor}\label{cor:geometric}
Let $\K$ be a perfect field, and let $X$ be an eventually coconnective noetherian DG-scheme over $\K$.
Suppose that the classical scheme underlying $X$ is an algebraic variety over $\K$.
Then $X$ is generically sequence-regular.
\end{cor}

\section{Special commutative DG-rings}\label{sec:special}

In classical commutative algebra, 
the four most prominent classes of commutative noetherian rings are the classes of regular rings, local complete intersection (lci) rings,
Gorenstein rings and Cohen–Macaulay rings.
It is well known that there is a chain of strict inclusions:
\begin{equation}\label{eqn:special}
\mbox{Regular rings} \subsetneq \mbox{Lci rings} \subsetneq \mbox{Gorenstein rings} \subsetneq \mbox{Cohen–Macaulay rings}. 
\end{equation}
See for example \cite[Page 171]{Mat}.
In this final section we discuss the corresponding situation in the DG setting.

The earliest generalization of these to the DG-setting was probably the notion of a Gorenstein DG-ring.
This notion was first introduced in topology by F\'{e}lix, Halperin and Thomas \cite{FHT}, in a 
context where the DG-rings are non-negatively graded.
Later, Avramov and Foxby studied Gorenstein homomorphisms in \cite{AvFox},
and this led them to the notion of a Gorenstein local DG-ring.
This notion was later generalized and extensively studied by Frankild, Iyengar, and J{\o}rgensen in \cite{FJ,FIJ}.
We will say that a commutative noetherian DG-ring $A$ with bounded cohomology is Gorenstein if $A_{\bar{\p}}$ has finite injective dimension over itself for all $\bar{\p} \in \opn{Spec}(\mrm{H}^0(A))$.

The next generalization which was introduced was a generalization of the notion of a local complete intersection rings.
These DG-rings are called quasi-smooth DG-rings. 
The first incarnation of this notion we are aware of was in the work \cite{ArGa} of Arinkin and Gaitsgory concerning the geometric Langlands conjecture. There, quasi-smooth DG-rings were studied over a fixed field,
and they were defined to be DG-rings whose cotangent complex has flat dimension at most $1$.

Cohen–Macaulay DG-rings were already discussed extensively above. 
They were introduced by the author in \cite{ShCM}.
It should be noted that unlike the other classes discussed here,
the Cohen–Macaulay property itself is not in general stable under localization,
so we distinguish between local-Cohen–Macaulay DG-rings and Cohen–Macaulay DG-rings,
which are defined to be those DG-rings which all of their localizations are local-Cohen–Macaulay.
It should be further noted, as demonstrated in \cite{ShKos}, 
that the theory of Cohen–Macaulay DG-rings is particularly well behaved, 
and similar to the classical theory, under the additional assumption that they have constant amplitude.

Finally, the generalization of regular rings we consider here is the topic of this paper, namely sequence-regular DG-rings.

We now discuss relations between these classes of DG-rings.
According to \cite[Corollary 2.2.8]{ArGa},
any quasi-smooth DG-ring is Gorenstein,
while by \cite[Proposition 8.9]{ShCM},
any Gorenstein DG-ring is Cohen–Macaulay.
It should be noted however that a Gorenstein DG-ring does not have to be of constant amplitude:
\begin{exa}
Let $A$ be any commutative noetherian local ring which has a dualizing complex,
but which is not equidimensional.
Let $R$ be a dualizing complex over $A$,
and shift it so that $\sup(R) < 0$.
Consider the trivial extension DG-ring $B = A \skewtimes R$,
as defined in \cite[Definition 1.2]{JorDual}.
According to \cite[Theorem 2.2]{JorDual},
the DG-ring $B$ is Gorenstein.
Letting $n = \inf(R)$,
by definition we have that $n = \inf(B)$,
and that $\mrm{H}^n(B) = \mrm{H}^n(R)$.
Moreover, $\mrm{H}^0(B) = A$.
Since $A$ is not equidimensional,
it follows from \cite[Remark 11.4.10]{ScSiBook} that $\opn{Supp}(\mrm{H}^n(R)) \subsetneq \opn{Spec}(A)$,
so that $B$ does not have constant amplitude.
\end{exa}

By Theorem \ref{thm:main} and its proof, 
we know that any sequence-regular DG-ring is Cohen–Macaulay and has constant amplitude.
Unfortunately, it turns out that sequence-regular DG-rings need not be quasi-smooth,
and they do not even need to be Gorenstein:
\begin{exa}
Let $\K$ be a field.
According to \cite[Theorem 5]{ShOpLoc},
there exist a commutative noetherian DG-ring $A$,
such that $A$ has bounded cohomology, 
$\mrm{H}^0(A) = \K$,
and such that $A$ is not Gorenstein.
On the other hand, the fact that $\mrm{H}^0(A)$ is a field implies by definition that $A$ is sequence-regular.
\end{exa}
It should be noted that the above is not an isolated example:
by Theorem \ref{thm:generic}, we know that if $A$ is a commutative noetherian DG-ring with bounded cohomology,
such that $A$ has a dualizing DG-module and $\mrm{H}^0(A)$ is sufficiently nice,
then $A$ is generically sequence-regular. 
On the other hand, by \cite[Theorem 5]{ShOpLoc}, for any commutative noetherian ring $B$ which has a dualizing complex,
there is a commutative noetherian DG-ring $A$ with bounded cohomology which has a dualizing DG-module,
such that $\mrm{H}^0(A) = B$, and such that $A$ is nowhere Gorenstein.

\begin{rem}
As noted above, quasi-smooth DG-rings are a generalization of locally complete intersection rings.
Recently, Pollitz gave in \cite[Theorem 5.2]{Pol} the following beautiful characterization of locally complete intersection rings:
a noetherian local ring $(A,\m)$ is a local complete intersection if and only if every object of $\cat{D}^{\mrm{b}}_{\mrm{f}}(A)$ is virtually small. Here, a non-zero object is called virtually small if the thick subcategory is generates contains a non-zero compact object.
This suggests that an alternative definition of a local complete intersection DG-ring is a noetherian local DG-ring $A$ for which every object of $\cat{D}^{\mrm{b}}_{\mrm{f}}(A)$ is virtually small.
We do not know if this definition leads to a different notion or whether it is equivalent to the notion of a quasi-smooth DG-ring.
In any case, we would like to mention that even if one works with this definition, 
it is still the case that not every sequence-regular DG-ring satisfies this definition.
Indeed, in \cite[Theorem 4.29]{ShWi},
there is an example of a commutative noetherian DG-ring $\K$ with bounded cohomology,
such that $\mrm{H}^0(\K)$ is a field (and hence, $\K$ is sequence-regular),
but there exist a non-zero object in $\cat{D}^{\mrm{b}}_{\mrm{f}}(\K)$ which is not virtually small.
\end{rem}

We summarize the above discussion in the following diagram,
which is a DG version of diagram (\ref{eqn:special}).
All inclusions in this diagram are strict.

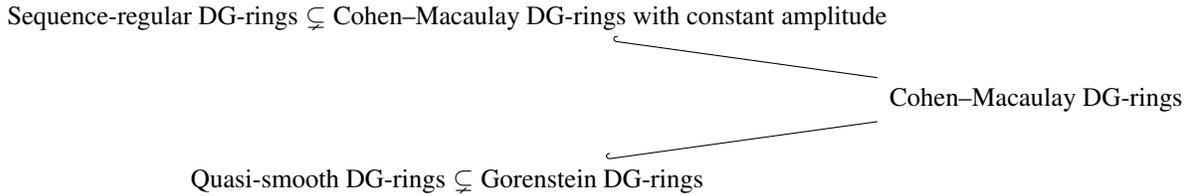
\[
\hspace*{-1.5cm}
\begin{tikzcd}[arrows = dash, row sep = 0.5cm,column sep=-0.8em]
\mbox{Sequence-regular DG-rings} \subsetneq \mbox{Cohen–Macaulay DG-rings with constant amplitude} \ar[hook]{rd} &\\
\phantom{\mbox{Sequence-regular DG-rings} \subsetneq \mbox{Cohen–Macaulay DG-rings with constant amplitude}} & \mbox{Cohen–Macaulay DG-rings}\\
\mbox{Quasi-smooth DG-rings}\subsetneq \mbox{Gorenstein DG-rings} \ar[hook]{ru} &
\end{tikzcd}
\]

\textbf{Acknowledgments.}

The author thanks Jordan Williamson for helpful discussions.
The author is thankful to an anonymous referee for suggestions that helped improving this
manuscript.
This work has been supported by the Charles University Research Centre program No.UNCE/SCI/022,
and by the grant GA~\v{C}R 20-02760Y from the Czech Science Foundation.

\bibliographystyle{plain}
\bibliography{references}

\end{document}